\documentclass[12pt,reqno]{amsart}

\usepackage{graphicx}
\usepackage{comment}
\usepackage[T1]{fontenc}
\usepackage{color}
\usepackage{hyperref}
\usepackage{amsmath, amsthm, amssymb, amsfonts}
\usepackage[mathscr]{eucal}
\usepackage[margin=2.32cm]{geometry}
\usepackage{verbatim}
\usepackage{enumitem}
\usepackage{chngcntr}
\newtheorem{theorem}{Theorem}[section]
\newtheorem{lemma}[theorem]{Lemma}

\theoremstyle{definition}

\numberwithin{subcase}{case}

\theoremstyle{definition}

\newtheorem{remark}[]{\textbf{Remark}}

\numberwithin{equation}{section}

\begin{document}
\title{Bilinear sums with $GL(2)$ coefficients and the exponent of distribution of $d_3$}
\author{Prahlad Sharma \vspace{-1cm}}
\address{Max Planck Institute for Mathematics, Vivatsgasse 7, 53111 Bonn}
\email{sharma@mpim-bonn.mpg.de}
\subjclass{11N25, 11N37, 11F41}
\maketitle	
\begin{abstract}
We obtain the exponent of distribution $1/2+1/30$ for the ternary divisor function $d_3$ to square-free and prime power moduli, improving the previous results of Fouvry--Kowalski--Michel, Heath-Brown, and Friedlander--Iwaniec. The key input is certain estimates on bilinear sums with $GL(2)$ coefficients obtained using the delta symbol approach.
\end{abstract}

\section{Introduction}
Given an arithmetically interesting function $f:\mathbb{N}\to \mathbb{C}$ and $q$ of reasonable size, we expect that
\begin{equation}\label{first}
\sum_{\substack{n\leq X\\n=a(q)}}f(n)\sim \frac{1}{\phi(q)}\sum_{\substack{n\leq X\\(n,q)=1}}f(n),
\end{equation}for each $(a,q)=1$. It is a fundamental problem in number theory to show the above asymptotic holds for $q$ as large as possible. To this end, we call a positive number $\delta$ an \textit{exponent of distribution for $f$ restricted to a set $\mathcal{Q}$ of moduli}, if for any $q\in\mathcal{Q}$ with $q\leq X^{\delta-\epsilon}$ and any residue class $a\pmod{q}$ with $(a,q)=1$, the asymptotic formula
\begin{equation}\notag
 \sum_{\substack{n\leq X\\n=a (q)}}f(n)=\frac{1}{\phi(q)}\sum_{\substack{n\leq X\\ (n,q)=1}}f(n)+O\left(\frac{X}{q(\log X)^A}\right)
 \end{equation} holds for any $A>0$ and $X\geq 2$.

For the very important Von Mangoldt function $\Lambda(n)$, the classical Siegel-Walfisz theorem implies that the above asymptotics hold for $q\leq (\log X)^{B(A)}$, where $B(A)>0$ depends on $A$, whereas the GRH predicts $q\leq X^{1/2-\epsilon}$. The celebrated Bombieri-Vinogradov theorem confirms this prediction on an average over the moduli.

Another important class of examples comes from the \textit{k-fold} divisor function
\begin{equation}\notag
d_k(n)=\sum_{n_1n_2\cdots n_k=n}1.
\end{equation}It is widely believed that $\delta=1$ is a exponent of distribution for all $k\geq 2$. This has deep consequences for our understanding of primes which goes far beyond the direct reach of the GRH. For $k=2$, the best known exponent of distribution is $\delta=2/3$ due to  Selberg (unpublished), Hooley \cite{hooley} and Heath-Brown \cite{hb}. Several authors have achieved improvement to $\delta=2/3$ in special cases. See \cite{blodiv, bhs, fouvryiwa, fouv}.

The only other case known for surpassing the `Bombieri-Vinogradov range' $\delta=1/2$ is when $k=3$. Let us briefly take a look at the previous approaches. After an application of the $GL(3)$ Voronoi summation formula to the left hand side of \eqref{first} (or equivalently, a three-fold application of the Poisson summation formula), one observes that to beat the $\delta_3=1/2$ barrier, one needs non-trivial estimates for
\begin{equation}\label{convo}
\sum_{m\sim q}d_3(m)\text{Kl}_3(am, q),
\end{equation}where $\text{Kl}_3(\cdots)$ is the hyper-Kloosterman sum. Opening the divisor function\\ $d_3(m)=\sum_{m_1m_2m_3=m}1$ and dividing the $m_i$-sum into dyadic blocks $m_i\sim Y_i$ with $Y_1\leq Y_2\leq Y_3$, it suffices to obtain non-trivial estimates for
\begin{equation}\label{decompi}
\sum_{m_1\sim Y}\sum_{m_2\sim q/Y}d(m_2)\text{Kl}_3(am_1m_2, q)
\end{equation}for each $Y\leq q^{1/3}$. When $Y$ is not too small, good estimates can be obtained by applying Cauchy-Schwarz inequality to \eqref{decompi} keeping the $m_2$ variable outside the absolute value square followed by a Poisson summation in the $m_2$-sum. Therefore, the main effort lies in obtaining good estimates for \eqref{decompi} when $Y$ is small. Alternatively, by applying the $GL(2)$ Voronoi summation formula to the $m_2$-sum in \eqref{decompi}, one can also consider
\begin{equation}\label{decompii}
\sum_{m_1\sim Y}\sum_{m_2\sim qY}d(m_2)e(am_1\overline{m_2}/q).
\end{equation}In their groundbreaking work, Friedlander and Iwaniec \cite{friediwa} successfully obtained non-trivial estimates for \eqref{decompii} which led to the exponent $\delta_3=1/2+1/230$. More precisely, their main input was non-trivial estimates for the short exponential sums
\begin{equation}\label{short}
\sum_{h\sim H}\sum_{m\sim M}\sum_{n\sim N}e(h\overline{mn}/q),
\end{equation}which is a further decomposition of \eqref{decompii}, using the ``shifting by $ab$'' technique. Heath-Brown \cite{hbd3} improved the exponent to $\delta_3=1/2+1/82$ by utilizing a more elementary treatment of \eqref{short} based on the methods of Heath-Brown \cite{hba} and Balasubramanian, Conrey, and Heath-Brown \cite{bala}. Since both of these approaches were based on decomposing the sum \eqref{decompi} into multiple exponential sums \eqref{short}, they were far from optimal.

With a more structural approach by viewing the divisor function $d(m_2)$ in \eqref{decompi} as the Fourier coefficients of Eisenstein series, Fouvry, Kowalski, and Michel \cite{fkm} were able to produce the exponent $\delta_3=1/2+1/46$ for prime moduli improving the previous results. Their key input was the estimates for short sums of $GL(2)$ coefficients 
\begin{equation}\label{shortgl2}
\sum_{m\sim M}\lambda(m)K(m)
\end{equation}twisted by general trace functions $K(\cdots)$ of prime modulus, which they obtained in \cite{fkm2} using the $GL(2)$ spectral theory. Note that the relevant estimate for \eqref{shortgl2} (when $\lambda(m)=d(m)$) was obtained in the separate paper \cite{fkm3}, which required additional arguments to isolate its contribution from the continuous spectrum. They further improved their exponent to $\delta_3=1/2+1/34$ on an average over the moduli by combining their results with the estimates for sums of Kloosterman sums pioneered by Deshouillers and Iwaniec. P. Xi \cite{ping} obtained the exponent $\delta_3=1/2+1/34$ for moduli with special factorisation using the $q$-analogue of the van der Corput method. 

In this paper, we go even further and utilise the complete bilinear structure in \eqref{decompi}, which results in an improvement over all the above exponents. We use the delta symbol approach to obtain non-trivial estimates for bilinear sums \eqref{decompi} involving $GL(2)$ coefficients. The method provides a uniform treatment for the holomorphic/Maass and Eisenstein cases and essentially covers all moduli.

\begin{theorem}\label{mainthm}
Let $\epsilon>0$ and $a$ be a non-zero integer. For every square-free $q\geq 1$ and every odd prime power $q=p^{\gamma}, \gamma\geq  28$ with $(a,q)=1$ and satisfying
 \begin{equation}\notag
 q\leq X^{1/2+1/30-\epsilon},
 \end{equation}we have
 \begin{equation}\notag
 \sum_{\substack{n\leq X\\n=a(q)}}d_3(n)=\frac{1}{\phi(q)}\sum_{\substack{n\leq X\\(n,q)=1}}d_3(n)+O(X^{1-\epsilon}/q),
 \end{equation}where the implied constant depends only on $\epsilon$.
\end{theorem}
\textbf{Remarks.}
\begin{itemize}
\item We have considered the complementary cases of square-free and prime power, but it is possible to merge the arguments to cover all natural numbers and, in particular, to close the gap from $\gamma\geq 28$ in Theorem \ref{mainthm} to $\gamma \geq 2$. All one needs is a version of Lemma \ref{ppowerchar} with restriction $u\leq 4\gamma/5$ lifted to $u\leq \gamma-1$ for $\gamma\geq 2$. The current estimation of the character sum leads to complicated counting arguments, which we decided to avoid for the exposition's simplicity.
\item For $q=p^{\gamma}$, the methods of this paper can produce a better exponent $\delta>1/2+1/30$ by using the $p$-adic stationary phase analysis followed by an exponent pair estimate (see \cite[Theorem 2]{milicevic}) to bound certain average of the product of two Kloosterman sums non-trivially. See the remarks just before Lemma \ref{altcharest} and Remark \ref{improv1}.
\item The exponents can be further improved by combining our estimates with the Kloostermania techniques when averaging over the moduli. 
\end{itemize}

The key input is the following estimate for the bilinear sums with $GL(2)$ coefficients obtained using the separation of oscillation technique. \\
\noindent
For $m,q\geq 1$, let $\tilde{\text{Kl}}_3(m, q)$ denote the normalised hyper-Kloosterman sum
\begin{equation}\label{hyper}
\tilde{\text{Kl}}_3(m,q)=\frac{1}{q}\,\,\sideset{}{^*}\sum_{\substack{x,y\,\, (q)}}e\left(\frac{mx+y+\overline{xy}}{q}\right).
\end{equation}Let $\lambda(n)$ denote the $n$-th Fourier coefficient of a $SL(2,\mathbb{Z})$ holomorphic cusp form or Maass cusp form or the Eisenstein series $E(z,1/2+w)$ for a complex number 
\begin{equation}\notag
w\ll q^{\epsilon}.
\end{equation}Note that in the case of Eisenstein series, $\lambda (n)=\sigma_{-2w}(n)$, which will be the relevant case for the application to Theorem \ref{mainthm}. We fix a smooth function $V(x)$ compactly supported in $\mathbb{R}_{>0}$ and satisfying $V^{(j)}(x)\ll_{j,\epsilon} q^{j\epsilon}, j\geq 0$.

\begin{theorem}\label{prime}
Let $q\geq 1$ be square-free, $b\in\mathbb{Z}$ co-prime to $q$ and $\lambda(n)$, $\tilde{\text{Kl}_3}(\cdots)$ be as above. Let  $\mathscr{N}\subset \mathbb{Z}$ be a set of $N$ consecutive integers and let $\{\alpha_n\}_{n\in \mathscr{N}}$ be a sequence of complex numbers with $|\alpha_n|\ll 1$. Suppose $M\geq 1$ is such that $N\leq q^{1/2}(1+M/q)^{-2}$, then
\begin{equation}\notag
\begin{aligned}
\sum_{n\in \mathscr{N}}\sum_{m\geq 1}\alpha_n\lambda(m)\tilde{\text{Kl}}_3(mnb, q)V(m/M)\ll_{\epsilon} q^{3/8+\epsilon}M^{1/2}N^{3/4}(1+M/q)^{1/2}&+q^{-1/4+\epsilon}MN^{3/2}(1+M/q)\\
&+Nq^{3/4+\epsilon}(1+M/q)^{1/2}.
\end{aligned}
\end{equation}
\end{theorem}The flexibility of the method allows us to obtain stronger estimates for moduli with special factorisation.
\begin{theorem}\label{primepower}
Let $q=p^{\gamma}$, $\gamma\geq 2$ and $p>2$. With the notations of Theorem \ref{prime}, suppose that $N\leq q^{1/5}(1+M/q)^{-2}$. Then 
\begin{equation}\notag
\sum_{n\in \mathscr{N}}\sum_{m\geq 1}\alpha_n\lambda(m)\tilde{\text{Kl}}_3(mnb, q)V(m/M)\ll_{\epsilon} p^{7/12}q^{1/3+\epsilon}M^{1/2}N^{5/6}(1+M/q)^{2/3}+q^{13/20+\epsilon}N.
\end{equation}
\end{theorem}
\textbf{Remarks.}
\begin{itemize}
\item Each $(1+M/q)$ factor that appears in last two theorems can be eliminated by first dualising the $GL(2)$ sum when $M\gg q$ and proceeding same as below.
The restriction $N\leq q^{1/5}(1+M/q)^{-2}$ in Theorem \ref{primepower} is for the sake of technical simplicity and can be easily extended to $N\leq q^{1/2}(1+M/p)^{-2}$ with additional computations.
\item The choice of the trace function $\text{Kl}_3(\cdots)$ in Theorem \ref{prime} and Theorem \ref{primepower} is made for its application towards Theorem \ref{mainthm}, but the results should hold for more general trace functions (see \cite{fkm2} for examples).
\item Choosing $\mathscr{N}=\{1\}$ and $\alpha_1=1$, the two theorems give
\begin{equation}\label{michel}
\sum_{m\sim M}\lambda(m)\tilde{\text{Kl}}_3(mb, q)\ll_{\epsilon} q^{\epsilon}\left( M^{1/2}q^{3/8}+M/q^{1/4}+q^{3/4}\right),
\end{equation}and
\begin{equation}\notag
\sum_{m\sim M}\lambda(m)\tilde{\text{Kl}}_3(mb, q)\ll_{\epsilon}q^{\epsilon}(M^{1/2}q^{1/3}+q^{13/20}).
\end{equation}when $M\leq q$. These are non-trivial as long as $M\gg q^{3/4+\epsilon}$, which is the `Burgess range', and $M\gg q^{2/3+\epsilon}$, which is the `Weyl range', respectively. Hence, with the additional cancellation in the $n$-sum, Theorem \ref{primepower} is, on average, of sub-Weyl strength (with twists by trace functions). In the square-free case, the `$Nq^{3/4}$' term in Theorem \ref{mainthm}, which pops out as an additional term from a certain zero-frequency, prevents us from going beyond Burgess. It would be of interest to get an improvement over this term.
\item For composite moduli $q=p_1p_2$, with $p_1\asymp q^{\alpha}$, $\alpha>0$ not too large, estimates somewhere between the Weyl and the Burgess range can be obtained using a similar approach.
\end{itemize}

\begin{remark}[Notation]
In this paper the notation $\alpha\ll A$ will mean that for any $\epsilon>0$, there is a constant $c$ such that $|\alpha|\leq cAX^{\epsilon}$. The dependence of the constant on $\epsilon$, when occurring, will be ignored. We will follow the usual $\epsilon$-convention : the letter $\epsilon$ denotes sufficiently small positive quantity that may change from line to line.  We will also use the phrase ``negligible error'' by which we mean an error term $O_{B}(X^{-B})$ for an arbitrary $B>0$. The notation $x=y(q)$ will mean $x\equiv y\bmod q$ throughout the paper.
\end{remark}

\section{Preliminaries}

\subsection{Voronoi summation formula for $d_3(n)$.}We use the version due to X. Li \cite{lidiv}. Let
\begin{equation}\label{ternary}
\sigma_{0,0}(k,l)=\sum_{d_1|l}\sum_{\substack{d_2| \frac{l}{d_1}\\ (d_2,k)=1}}1=\sum_{a|(k,l)}\mu(a)d_3(l/a).
\end{equation}For $\phi(y)\in C_{c}(0,\infty), k=0,1$ and $\sigma>-1-2k$, set
\begin{equation}\label{Phi}
\Phi_k(y)=\frac{1}{2\pi i}\int_{(\sigma)}(\pi^3y)^{-s}\cdot \frac{\Gamma\left(\frac{1+s+2k}{2}\right)^3}{\Gamma\left(\frac{-s}{2}\right)^3}\tilde{\phi}(-s-k)\,ds,
\end{equation}where $\tilde{\phi}$ is the Mellin transform of $\phi$, and
\begin{equation}\label{Phisplit}
\Phi_{\pm}(y)=\Phi_0(y)\pm \frac{1}{i\pi^3y}\Phi_1(y).
\end{equation}
\begin{lemma}[X. Li]\label{vord3}
For integers $a,q\geq 1$ with $(a,q)=1$, with the above notation, we have
\begin{equation}\notag
\begin{aligned}
&\sum_{n\geq 1}d_3(n)e(an/q)\phi(n)\\
&=\frac{1}{q}\int_{0}^{\infty}P(\log y, q)\phi(y)dy\\
&+\frac{q}{2\pi^{3/2}}\sum_{\pm}\sum_{r|q}\sum_{m\geq 1}\frac{1}{rm}\sum_{r_1|r}\sum_{r_2|\frac{r}{r_1}}\sigma_{0,0}(r/(r_1r_2), m)S(\pm m, \overline{a}; q/r)\Phi_{\pm}(mr^2/q),
\end{aligned}
\end{equation}where $P(y,q)=A_0(q)+A_1(q)y+A_2(q)y^2$ is a quadratic polynomial whose coefficients depend only on $q$.
\end{lemma}When $\phi$ is a nice weight function the corresponding transform $\Phi_{\pm}$ also behaves nicely as conveyed by the following lemma.
\begin{lemma}\label{Phiprop}
Suppose the smooth function $\phi(y)$ is supported in $[X,2X], X\geq 1$ and satisfies $y^j\phi^{(j)}(y)\ll_j 1, j\geq 0$. Then $\Phi_{\pm}(y)\ll_{B}X^{-B}$ unless  $yX\ll X^{\epsilon}$ in which case
\begin{equation}\notag
y^j\Phi^{(j)}_{\pm}(y)\ll_j \min\{yX, 1\}.
\end{equation}
\end{lemma}	
\begin{proof}
From \eqref{Phisplit}, it is enough to prove the lemma for $\Phi_0(y)$. From the definition \eqref{Phi} we have for $j\geq 0$,
\begin{equation}\notag
y^{j}\Phi_0^{(j)}(y)=(-1)^j\frac{1}{2\pi i}\int_{(\sigma)}s(s+1)\cdots (s+j)(\pi^3y)^{-s}\cdot \frac{\Gamma\left(\frac{1+s}{2}\right)^3}{\Gamma\left(\frac{-s}{2}\right)^3}\tilde{\phi}(-s)\,ds.
\end{equation}Note that $\tilde{\phi}(-s)\ll X^{-\Re(s)}$. Shifting the contour above to the right $\sigma=A>0$ and trivially estimating we obtain
\begin{equation}\label{fp}
y^{j}\Phi_0^{(j)}(y)\ll_{A,j} (yX)^{-A}.
\end{equation}Since $A$ is arbitrary, the first part of lemma follows. On the other hand, shifting the contour to the left $\sigma=-3/2$ while picking up the residue at $\sigma=-1$ we obtain
\begin{equation}\label{sp}
y^{j}\Phi_0^{(j)}(y)\ll_j yX+(yX)^{3/2}.
\end{equation}The second part of the lemma from \eqref{fp} and \eqref{sp}.
\end{proof}

\subsection{Voronoi summation formula for $GL(2)$.}See appendix A.4 of \cite{kmv} and appendix of \cite{kmvtau} for details.
\begin{lemma}\label{gl2vor}
Let $\lambda(n)$ be either the $n$-th Fourier coefficient of a Maass cusp form with Laplacian eigenvalue $1/4+\nu^2$, $\nu\geq 0$, or $\lambda(n)=\sigma_{w}(n), w\in\mathbb{C}$. For integers $a,q\geq 1$ with $(a,q)=1$, $h(x)\in C_{c}(0,\infty)$, we have
\begin{equation}\notag
\begin{aligned}
\sum_{n=1}^{\infty}\lambda(n)e\left(\frac{an}{q}\right)h(n)=\frac{1}{q}\int_{0}^{\infty}g(q,x) h(x)\,dx+\frac{1}{q}\sum_{\pm}\sum_{n\geq 1}\lambda(n)e\left(\frac{\pm\overline{a}n}{q}\right)H^{\pm}\left(\frac{n}{q^2}\right),
\end{aligned}
\end{equation}where
\begin{itemize}
\item if $\lambda(n)$ corresponds to Maass form, then $g(q,x)=0$ and
\begin{equation*}
\begin{aligned}
& H^{-}(\alpha)=\frac{-\pi}{\sin (\pi i\nu)}\int_{0}^{\infty}h(y)\{J_{2i\nu}-J_{-2i\nu }\}(4\pi\sqrt{y\alpha})dy \,,\\
& H^{+}(\alpha)=4\varepsilon_f\cosh (\pi\nu)\int_{0}^{\infty}h(y)K_{2i\nu}(4\pi\sqrt{y\alpha})dy\,,
\end{aligned}
\end{equation*}for $\nu>0$, and 
\begin{equation*}
H^{-}(\alpha)=-2\pi\int_{0}^{\infty}h(y)Y_0(4\pi\sqrt{y\alpha})dy,\,\,\,\text{and}\,\,\, H^{+}(\alpha)=4\varepsilon_f\int_{0}^{\infty}h(y)K_{0}(4\pi\sqrt{y\alpha})dy\,,
\end{equation*}for $\nu=0$.
\item If $\lambda(n)=\sigma_{0}(n)=d(n)$, then $g(q,x)=\log(\sqrt{x}/q)+\gamma$ and
\begin{equation}\notag
\begin{aligned}
& H^{-}(\alpha)=-2\pi\int_{0}^{\infty} h(y)Y_0\left(4\pi\sqrt{y\alpha}\right)\,dy \,,\\
& H^{+}(\alpha)=4\int_{0}^{\infty}h(y)K_{0}(4\pi\sqrt{y\alpha})dy\,.
\end{aligned}
\end{equation}
\item If  $\lambda(n)=\sigma_{w}(n), w\neq 0$, then $g(q,x)=\zeta(1+w)(x/q)^w+\zeta(1-w)q^{w}$ and
\begin{equation}\notag
\begin{aligned}
& H^{-}(\alpha)=\int_{0}^{\infty} h(y)\tilde{Y}_{w}\left(4\pi\sqrt{y\alpha}\right)\,dy \,,\\
& H^{+}(\alpha)=\int_{0}^{\infty}h(y)\tilde{K}_{w}(4\pi\sqrt{y\alpha})dy\,,
\end{aligned}
\end{equation}where $\tilde{Y}_w, \tilde{K}_{w}$ are closely related to $Y_w, K_w$, and have the integral representations
\begin{equation}\notag
\begin{aligned}
&\tilde{Y}_{w}(x)=\frac{1}{2\pi i}\int_{(2)}(x/2)^{-s}\Gamma(s-w)\Gamma(s+w) \cos (\pi s)ds,\\
&\tilde{K}_{w}(x)=\frac{\cosh (\pi|w|)}{2\pi i}\int_{(2)}(x/2)^{-s}\Gamma(s-w)\Gamma(s+w) ds.
\end{aligned}
\end{equation}
\end{itemize}
\end{lemma}
\subsection{Character sum estimates}The endgame of the paper consists of getting square-root cancellations in certain character sums which we record here for convenience. Let $p$ be a prime and $m\in\mathbb{Z}$. Suppose $s_j,t_j,\lambda_j, j=1,2$ are integers such that $(s_j,p)=(\lambda_j,p)=1, j=1,2$. For $u\leq \gamma $, define
\begin{equation}\label{charrecall0}
\mathfrak{C}_{\gamma,u}=\mathop{\sideset{}{^*}\sum\sideset{}{^*}\sum}_{\substack{a_1,a_2 (p^u)\\\lambda_1\overline{a_1}-\lambda_2\overline{a_2}=m (p^u)}}S(1, \overline{s_1p^{\gamma-u}a_1+t_1}, p^{\gamma})\overline{S}(1, \overline{s_2p^{\gamma-u}a_2+t_2}, p^{\gamma}).
\end{equation}For $\gamma=1$, such character sums has been studied in \cite{df} using the $l$-adic techniques developed by Deligne and Katz, and in \cite{fkm2} in the broader framework of trace functions. When $\gamma>1$, an estimate for $\mathfrak{C}_{\gamma,u}$ can be obtained in an elementary manner by reducing the sum to a set of congruence conditions. We begin with latter case.
\begin{lemma}\label{ppowerchar}
Suppose $\gamma>1, u\leq 4\gamma/5, m\neq 0$ and $(2t_j,p)=1$. If $u/2< \gamma-u$ or $\nu_p(m)<\gamma-u$, then
\begin{equation}\notag
\mathfrak{C}_{\gamma,u}\ll p^{\gamma+u/2+\epsilon(u)/2}\cdot p^{\nu_p(m)},
\end{equation}and if $u/2\geq \gamma-u$ and $\nu_p(m)\geq \gamma-u$, then $\mathfrak{C}_{\gamma,u}$ vanishes unless $t_1^{-3/2}s_1\lambda_1=t_2^{-3/2}s_2\lambda_2\,\,({p^{\gamma-u}})$, in which case
\begin{equation}\notag
\mathfrak{C}_{\gamma,u}\ll p^{\gamma+u}.
\end{equation}Here $\epsilon(u)=0$ or $1$ depending on $u$ is even or odd respectively.
\end{lemma}	
\begin{proof}
Without loss of generality we can assume $\nu_p(m)<u/2$, since otherwise the claim follows after a trivial estimation of the Kloosterman sums. We perform some initial transformation. First suppose $u$ is even. Then for $j=1,2$, we can write
\begin{equation}\label{aidecomp0}
a_j=p^{u/2}\alpha_j+\beta_j,\,\,1\leq\alpha_j,\beta_j\leq p^{u/2},\,\,(\beta_j,p)=1.
\end{equation}From \eqref{aidecomp0} we obtain
\begin{equation}\label{splitting0}
\overline{a_j}=\overline{\beta_j}-p^{u/2}\overline{\beta_j}^2\alpha_j\,\,({p^{u}}).
\end{equation}Plugging \eqref{splitting0} we see that the congruence
\begin{equation}\notag
\lambda_1\overline{a_1}-\lambda_2\overline{a_2}=m \,\,({p^u}),
\end{equation}is equivalent to 
\begin{equation}\label{alpha20}
\begin{aligned}
&\overline{\beta_2}=\lambda_1\overline{\lambda_2}\overline{\beta_1}-\overline{\lambda_2}m\,\,({p^{u/2}}),\\
&\alpha_2=\lambda_1\overline{\lambda_2}\beta_2^2\overline{\beta_1}^2\alpha_1-g(\beta_1)\,\, ({p^{u/2}}),
\end{aligned}
\end{equation}where 
\begin{equation}\notag
g(\beta_1)=\overline{\lambda_2}\beta_2^2\cdot\frac{(\lambda_1\overline{\beta_1}-\lambda_2\overline{\beta_2}-m)}{p^{u/2}}.
\end{equation}We proceed for the explicit evaluation of $\mathfrak{C}_{\gamma, u}$ \eqref{charrecall0} in terms of these decomposition. We use the following evaluation of the Kloosterman sums modulo prime powers, which can be found in \cite[(12.39)]{iwaniec} :
\begin{equation}\label{kloosterman0}
S(1,\beta, p^{\gamma})=
\begin{cases}
2\left(\frac{\ell}{p}\right)^{\gamma}p^{\gamma/2}\Re\,\varepsilon_{p^{\gamma}}\,e(2\ell/p^{\gamma}), & \left(\frac{\beta}{p}\right)=1,\\
0, & \left(\frac{\beta}{p}\right)=-1,
\end{cases}
\end{equation}where $\ell^2=\beta \,\,(p^{\gamma})$, $\left(\frac{\cdot}{\cdot}\right)$ is the Legendre symbol, and $\varepsilon_{c}$ equals 1 if $c\equiv 1\bmod 4$ and $i$ if $c\equiv 3\bmod 4$. \\
\\
Hence, the Kloosterman sums in \eqref{charrecall0} vanishes unless we have $\left(\frac{t_j}{p}\right)=1$. From the formula \eqref{kloosterman0} it follows,
\begin{equation}\label{evalkloos}
S(1, \overline{s_jp^{\gamma-u}a_j+t_j}, p^{\gamma})=\sum_{\pm}p^{\gamma/2}\left(\frac{t_j^{1/2}}{p}\right)^{\gamma}e\left(\pm\frac{2\overline{(s_jp^{\gamma-u}a_j+t_j)^{1/2}}}{p^{\gamma}}\right).
\end{equation}Using the fact that $\gamma-u\geq 1$ and expanding $\overline{(s_jp^{\gamma-u}a_j+t_j)^{1/2}}$, we see that our character sum \eqref{charrecall0} can be written as sum of four terms of the form (upto constant factors)
\begin{equation}\label{upf}
\begin{aligned}
\mathfrak{C}=p^{\gamma}\left(\frac{t_1^{1/2}}{p}\right)^{\gamma}\left(\frac{t_2^{1/2}}{p}\right)^{\gamma}\mathop{\sideset{}{^*}\sum\sideset{}{^*}\sum}_{\substack{a_1,a_2 (p^u)\\\lambda_1\overline{a_1}-\lambda_2\overline{a_2}=m (p^u)}}e\left(\frac{\sum_{i\geq 0}p^{i(\gamma-u)}\theta_ia_1^{i}-\sum_{i\geq 0}p^{i(\gamma-u)}\eta_ia_2^i}{p^{\gamma}}\right),
\end{aligned}
\end{equation}where
\begin{equation}\notag
\theta_i=2\binom{-1/2}{i}t_1^{-i-1/2}s_1^{i}\,\,\,\,\text{and}\,\,\,\,\,\eta_i=2\binom{-1/2}{i}t_2^{-i-1/2}s_2^{i}.
\end{equation}Here $$\binom{-1/2}{i}=\frac{\left(-\frac{1}{2}\right)\left(-\frac{1}{2}-1\right)\cdots \left(-\frac{1}{2}-i+1\right)}{i!}$$ is the $i^{th}$ binomial coefficient, and in our context $-1/2$ means $-\overline{2} (p^{\gamma})$. Note that this way the numerator of this $i^{th}$ binomial coefficient is divisible by $i!$ and so the expression makes sense modulo $p^{\gamma}$.

Using \eqref{splitting0}, modulo $p^{\gamma}$, the phase function above is
\begin{equation}\notag
\begin{aligned}
\sum_{i\geq 0}p^{i(\gamma-u)}\theta_ia_1^{i}-\sum_{i\geq 0}p^{i(\gamma-u)}\eta_ia_2^i&=\sum_{i\geq 1}p^{i(\gamma-u)+u/2}i\theta_i\beta_1^{i-1}\alpha_1-\sum_{i\geq 1}p^{i(\gamma-u)+u/2}i\eta_i\beta_2^{i-1}\alpha_2\\
&\hspace{3cm}+\sum_{i\geq 0}p^{i(\gamma-u)}\theta_i\beta_1^i-\sum_{i\geq 0}p^{i(\gamma-u)}\eta_i\beta_2^i\\
&=\sum_{i=1,2}i\left(p^{i(\gamma-u)+u/2}i\theta_i\beta_1^{i-1}\alpha_1-p^{i(\gamma-u)+u/2}i\eta_i\beta_2^{i-1}\alpha_2\right)\\
&\hspace{3cm}+\sum_{i\geq 0}p^{i(\gamma-u)}\theta_i\beta_1^i-\sum_{i\geq 0}p^{i(\gamma-u)}\eta_i\beta_2^i.
\end{aligned}
\end{equation}We have truncated the last sum upto $i\leq 2$ since $3(\gamma-u)+ u/2\geq \gamma$ by our assumption. Substituting $\alpha_2$ from \eqref{alpha20}, the right hand side of the last display becomes
\begin{equation}\notag
\begin{aligned}
\alpha_1\sum_{i=1,2}p^{i(\gamma-u)+u/2}i\left(\theta_i-\eta_i\lambda_1\overline{\lambda_2}(\overline{\beta_1}\beta_2)^{i+1}\right)\beta_1^{i-1}&-\sum_{i=1,2}p^{i(\gamma-u)+u/2}i\eta_i\beta_2^{i-1}g(\beta_1)\\
&+\sum_{i\geq 0}p^{i(\gamma-u)}(\theta_i\beta_1^i-\eta_i\beta_2^i).
\end{aligned}
\end{equation}Substituting this expansion into \eqref{upf}, we see that
\begin{equation}\label{linear}
\mathfrak{C}= p^{\gamma}\left(\frac{t_1^{1/2}}{p}\right)^{\gamma}\left(\frac{t_2^{1/2}}{p}\right)^{\gamma}\sideset{}{^*}\sum_{1\leq\beta_1\leq p^{u/2}}e\left(\frac{f(\beta_1)}{p^{\gamma}}\right)\sum_{1\leq\alpha_1\leq p^{u/2}}e\left(\frac{h(\beta_1)\alpha_1}{p^{u/2}}\right),
\end{equation}where
\begin{equation}\notag
f(\beta_1)=-\sum_{i=1,2}p^{i(\gamma-u)+u/2}i\eta_i\beta_2^{i-1}g(\beta_1)+\sum_{i\geq 0}p^{i(\gamma-u)}(\theta_i\beta_1^i-\eta_i\beta_2^i)
\end{equation}and
\begin{equation}\label{hbeta}
h(\beta_1)=\sum_{i=0,1}(i+1)\beta_1^{i}\left(\theta_{i+1}-\eta_{i+1}\lambda_1\overline{\lambda_2}(\overline{\beta_1}\beta_2)^{i+2}\right)p^{i(\gamma-u)}.
\end{equation}Executing the linear $\alpha_1$-sum, it follows 
\begin{equation}\notag
\mathfrak{C}\ll p^{\gamma+u/2}\sideset{}{^*}\sum_{\substack{\beta_1 (p^{u/2})\\ h(\beta_1)=0 (p^{u/2})}}1.
\end{equation}It remains to count the solutions to $h(\beta_1)=0 (p^{u/2})$. 

If $\gamma-u> u/2$, then from the expression \eqref{hbeta} it follows that $h(\beta_1)=0 (p^{u/2})$ implies
\begin{equation}\notag
\theta_1-\eta_1\lambda_1\overline{\lambda_2}(\overline{\beta_1}\beta_2)^{2}=0\,\,({p^{u/2}})\Rightarrow \theta_1-\eta_1\lambda_1\overline{\lambda_2}( \lambda_1\overline{\lambda_2}-\overline{\lambda_2}m\beta_1)^{-2}=0 \,\, ({p^{u/2}}).
\end{equation}Since $(\theta_1\eta_1,p)=1$, the last relation forces $\overline{\theta_1}\eta_1\lambda_1\overline{\lambda_2}$ to be a quadratic residue $\bmod\,\,{p^{u/2}}$ in which case, we get
\begin{equation}\notag
\overline{\lambda_2}m\beta_1=\lambda_1\overline{\lambda_2}\pm (\overline{\theta_1}\eta_1\lambda_1\overline{\lambda_2})^{1/2} \,\, ({p^{u/2}}).
\end{equation}This determines $\beta_1$ modulo $p^{u/2-\min\{u/2,\nu_p(m)\}}$ and hence we have at most $O(p^{\nu_p(m)})$ solutions for $\beta_1 (p^{u/2})$ and the lemma follows. 

So we can assume $\gamma-u\leq u/2$. This forces 
\begin{equation}\label{firstc}
\theta_1-\eta_1\lambda_1\overline{\lambda_2}(\overline{\beta_1}\beta_2)^{2}=0\,\,({p^{\gamma-u}})\Rightarrow \theta_1-\eta_1\lambda_1\overline{\lambda_2}( \lambda_1\overline{\lambda_2}-\overline{\lambda_2}m\beta_1)^{-2}=0\,\, ({p^{\gamma-u}}).
\end{equation}Suppose $\nu_p(m)\geq\gamma-u (\geq u/4)$, then the last congruence becomes
\begin{equation}\notag
\theta_1\lambda_1=\eta_1\lambda_2\,\, ({p^{\gamma-u}}),\,\text{i.e.}\,\,t_1^{-3/2}s_1\lambda_1=t_2^{-3/2}s_2\lambda_2\,\,({p^{\gamma-u}}).
\end{equation}In this we case we use the trivial bound to get
\begin{equation}\notag
\mathfrak{C}\ll  p^{\gamma+u/2}\sideset{}{^*}\sum_{\substack{\beta_1 (p^{u/2})\\ h(\beta_1)=0 (p^{u/2})}}1 \ll p^{\gamma+u}.
\end{equation}
This proves the second part of the lemma. In the case $\nu_p(m)<\gamma-u$, \eqref{firstc} determines $\beta_1$ modulo $p^{\gamma-u-\nu_p(m)}$, and say $c$ is the corresponding solution . Denote $r=\gamma-u-\nu_p(m), h_i(c)=(i+1)(\theta_{i+1}-\eta_{i+1}\lambda_1\overline{\lambda_2}( \lambda_1\overline{\lambda_2}-\overline{\lambda_2}mc)^{-i-2})$, then for $\lambda\in\mathbb{Z}$,
\begin{equation}\notag
h(p^r\lambda+c)=h_0(c)-\eta_0\lambda_1\overline{\lambda_2}^2( \lambda_1\overline{\lambda_2}-\overline{\lambda_2}mc)^{-3}mp^r\lambda+p^r\lambda h_1(c)p^{\gamma-u}+ch_1(c)p^{\gamma-u} \,\, ({p^{u/2}}).
\end{equation}Dividing the right hand side by $p^{\gamma-u}$, $h(p^r\lambda+c)=0 \,\, ({p^{u/2}})$ boils down to
\begin{equation}\notag
\lambda(h_1(c)p^r-\eta_0\lambda_1\overline{\lambda_2}^2( \lambda_1\overline{\lambda_2}-\overline{\lambda_2}mc)^{-3}(mp^r/p^{\gamma-u}))+h_0(c)/p^{\gamma-u}+ch_1(c)=0 \,\, ({p^{u/2-(\gamma-u)}}).
\end{equation}Note that the coefficient attached to $\lambda$ is coprime to $p$ and consequently $\lambda$ is determined modulo $p^{u/2-(\gamma-u)}$. Combining, it follows that $\beta_1$ is determined modulo $p^{u/2-\nu_p(m)}$ and therefore
\begin{equation}\notag
\mathfrak{C}\ll p^{\gamma+u/2}\sideset{}{^*}\sum_{\substack{\beta (p^{u/2})\\ h(\beta_1)=0 (p^{u/2})}}1\ll p^{\gamma+u/2}p^{\nu_p(m)},
\end{equation}which is the first part of the lemma. 

This completes the proof of the lemma when $u$ is even. When $u$ is odd, in \eqref{aidecomp0} we decompose the $a_j's$  as $p^{(u+1)/2}\alpha_j+\beta_j, 1\leq\alpha_j\leq p^{(u-1)/2}, 1\leq\beta_j\leq p^{(u+1)/2}$ and proceed identically as above. This way we gain a $p^{1/2}$ factor in the $\alpha_1$-sum but loose a factor of $p$ in the $\beta_1$-sum since the linear sum $\alpha_1$-sum in \eqref{linear} will now only determine $h(\beta_1)$ modulo $p^{(u-1)/2}$. This will result in an extra factor of $p^{1/2}$ the final estimate as indicated in the statement of the lemma.
\end{proof}
\begin{lemma}\label{primecharest}
With the notations of \eqref{charrecall0}, we have
\begin{equation}\notag
\mathfrak{C}_{1,1}=\mathop{\sideset{}{^*}\sum\sideset{}{^*}\sum}_{\substack{a_1,a_2 (p)\\\lambda_1\overline{a_1}-\lambda_2\overline{a_2}=m (p)}}S(1, \overline{s_1a_1+t_1}, p)\overline{S}(1, \overline{s_2a_2+t_2}, p)\ll p^{3/2}+p^2\delta_{\left(\substack{m=0 (p)\\t_1=t_2 (p)\\\lambda_1s_1=\lambda_2s_2 (p)}\right)}.
\end{equation}
\end{lemma}	
\begin{proof}
Consider the linear transformations
\begin{equation}\notag
\delta_1=\begin{pmatrix}
0 & 1 \\
s_1 & t_1
\end{pmatrix},\,
\delta_2=\begin{pmatrix}
0 & 1 \\
s_2 & t_2
\end{pmatrix}\,\,\text{and}\,\,
\delta_3=\begin{pmatrix}
\lambda_2 & 0\\
-m & \lambda_1
\end{pmatrix}.
\end{equation}Then we can recast $\mathfrak{C}_{1,1}$ as
\begin{equation}\notag
\mathfrak{C}_{1,1}=\sideset{}{^*}\sum_{a_1(p)}S(1, \delta_1(a_1), p)\overline{S}(1, \delta_2\delta_3(a_1),p).
\end{equation}Note that $\det (\delta_i)\neq 0\,\, (p)$. Hence
\begin{equation}\label{charreducf0}
\mathfrak{C}_{1,1}=\sideset{}{^*}\sum_{a_1(p)}S(1, a_1, p)\overline{S}(1, \delta_2\delta_3\delta_1^{-1}(a_1),p).
\end{equation}We are in position to apply the estimates from \cite{df}. Their Proposition 3.3 and Proposition 3.4 amounts to the following. Given $\delta=\begin{pmatrix}a & b\\ c & d\end{pmatrix}$ such that $ad-bc\neq 0\,\, (p)$, then
\begin{equation}\label{dfest}
\sideset{}{^*}\sum_{\alpha(p)}S(1,\alpha; p)\overline{S}(1, \gamma(\alpha); p)\ll p^{3/2}+p^2\delta_{(a-d=b=c=0 (p))}.
\end{equation} In our situation \eqref{charreducf0} we have
 \begin{equation}\notag
\delta_2\delta_3\delta_1^{-1}=s_1^{-1}\begin{pmatrix}
mt_1+\lambda_1s_1 & -m\\
s_1t_2\lambda_1-t_1(s_2\lambda_2-t_2m) & s_2\lambda_2-t_2m
\end{pmatrix}.
 \end{equation}The relation $a-d=b=c=0\,\, (p)$ from the \eqref{dfest} translates into $m=0\,\, (p), \lambda_1s_1=\lambda_2s_2\,\, (p)$ and $t_1=t_2\,\, (p)$. The lemma follows.
\end{proof}

\section{Proof of Theorem \ref{prime}}
Here $q$ is a square-free number and $(q,b)=1$. We are interested in
\begin{equation}\label{S}
S=\sum_{n\in\mathscr{N}}\sum_{m\geq 1}\alpha_n\lambda(m)\tilde{\text{Kl}}_3(mnb,q)V(m/M).
\end{equation}Without loss of generality, we can assume the sum over $n$ above is restricted to $(n,q)=1$, for if $(n,q)=d$, the hyper-Kloosterman sum degenerates to
\begin{equation}\notag
\tilde{\text{Kl}}_3(mnb,q)=\frac{d}{q}\cdot \tilde{\text{Kl}}_3(m(n/d)b\overline{d},q/d),
\end{equation}with which one arrives at a sum similar to \eqref{S} with a smaller modulus $q/d$ and a smaller $n$-sum with length $N/d$. Hence it is enough to consider
\begin{equation}\notag
S=\sideset{}{^*}\sum_{n\in\mathscr{N}}\sum_{m\geq 1}\alpha_n\lambda(m)\tilde{\text{Kl}}_3(mnb,q)V(m/M),
\end{equation}where the `$*$' over the $n$-sum denotes $(n,q)=1$. For simplicity we set $$K(m)=\tilde{\text{Kl}}_3(mb,q).$$ 

We begin by separating the coefficients $\lambda(m)$ and $K(mn)$ using the delta symbol. Due to structural reasons, the sizes of the moduli appearing in the delta expansion play no essential role in our approach and only act as a set of auxiliary variables. Hence we do not require any non-trivial delta symbol expansion and simply use the additive characters with large moduli. This simplifies many of the forthcoming calculations. This is not a new observation and was previously exploited in \cite{alhs} in the context of the subconvexity problem for $GL(2)$. 

Next, we note that a direct application of the delta symbol fails to beat the trivial bound at a certain diagonal contribution. To overcome this, we consider an amplified version of $S$ that introduces more harmonics into the analysis. 
Let $L\geq 1$, which will be chosen later, and  $\mathscr{L}$ be the set of primes in $[L,2L]$ co-prime to $q$. Note that
\begin{equation}\label{pntgl2}
\sum_{\ell \in \mathscr{L}}|\lambda(\ell)|^2=\sum_{\substack{\ell \sim L, \ell \,\,\,\text{prime}\\ (\ell, q)=1}}|\lambda(\ell)|^2=\sum_{\substack{\ell \sim L \\\ell \,\,\,\text{prime} }}|\lambda(\ell)|^2-\sum_{\substack{p|q\\p\sim L}}|\lambda(p)|^2\sim L+O(q^{\epsilon}L^{14/64}),
\end{equation}using the $GL(2)$ prime number theorem and the Kim-Sarnak bound for individual $GL(2)$ coefficients. Hence using the Hecke relation
\begin{equation}\notag
\lambda(\ell)\lambda(m)=\lambda(m\ell)+\lambda(m/\ell)\delta_{\ell |m},
\end{equation}and the asymptotic \eqref{pntgl2} we see that
\begin{equation}\label{amp}
|S|\ll |\tilde{S}|+O(MN/L),
\end{equation}where
\begin{equation}\notag
\tilde{S}=\frac{1}{L}\sum_{\substack{\ell\in\mathscr{L}\\}}\overline{\lambda(\ell)}\sideset{}{^*}\sum_{n\in\mathscr{N}}\sum_{m\geq 1}\alpha_n\lambda(m\ell)K(mn)V(m/M).
\end{equation}We have used the Ramanujan bound on average $\sum_{n\leq x}|\lambda(n)|^2\ll x$ and the well known Deligne's estimate $K(m)\ll 1$ for the last assertion. The rest of this section is devoted to the estimation of $\tilde{S}$.  Let $\mathscr{C}$ be the set primes in $[C,2C]$, with $C$ such that
\begin{equation}\notag
qC>100ML.
\end{equation}Since there is no restriction on the upper bound for $C$, a suitable large $C$ will ensure that $(c,q\ell)=1$ for all $ c\in\mathscr{C},\ell \in \mathscr{L}$. Due to the above inequality we can write $\tilde{S}$ as 
\begin{equation}\notag
\begin{aligned}
\tilde{S}=\frac{1}{CL}\sum_{\ell\in\mathscr{L}}\overline{\lambda(\ell)}\sum_{c\in\mathscr{C}}\sideset{}{^*}\sum_{n\in\mathscr{N}}\alpha_n\mathop{\sum\sum}_{\substack{m_1,m_2\geq 1\\qc|m_1-m_2\ell}}\lambda(m_1) K(m_2n)V(m_1/(M\ell))&V_1(m_2/M)\\
&\times e(q^{\epsilon}(m_1-m_2\ell)/ML)
\end{aligned},
\end{equation}where $V_1$ is another smooth function compactly supported in $\mathbb{R}_{>0}$ such that $V_1(x)=1, x\in \text{supp}(V)$. The artificial twist by $e(p^{\epsilon}(m_1-m_2\ell)/ML)$ allows us keep the length of the dual sums in their generic range. This turns out to be crucial in certain counting arguments of the paper, especially when $M\gg q$. The additional restriction modulo $p$ in $pc|(m_1-m_2\ell)$ acts as a conductor lowering mechanism. Detecting the congruence condition using additive characters we obtain
\begin{equation}\notag
\begin{aligned}
\tilde{S}=\frac{1}{qCL}\sideset{}{^*}\sum_{n\in\mathscr{N}}\alpha_n\sum_{\ell\in\mathscr{L}}\overline{\lambda(\ell)}\sum_{c\in\mathscr{C}}\frac{1}{c}\sum_{a(qc)}\mathop{\sum\sum}_{\substack{m_1,m_2\geq 1}}&\lambda(m_1)K(m_2n)e(a(m_1-m_2\ell)/qc)\\
&\times V(m_1/(M\ell))V_1(m_2/M)e(p^{\epsilon}(m_1-m_2\ell)/ML).
\end{aligned}
\end{equation}Since $(c,q)=1$, we can split the above sum as
\begin{equation}\label{ssplit}
\tilde{S}=\sum_{d|q}S(d)+\mathscr{S}
\end{equation}where
\begin{equation}\label{S_d}
\begin{aligned}
S(d)=\frac{1}{qCL}\sideset{}{^*}\sum_{n\in\mathscr{N}}\alpha_n\sum_{\ell\in\mathscr{L}}\overline{\lambda(\ell)}\sum_{c\in\mathscr{C}}\frac{1}{c}\sideset{}{^*}\sum_{a(dc)}\mathop{\sum\sum}_{\substack{m_1,m_2\geq 1}}&\lambda(m_1)K(m_2n)e(a(m_1-m_2\ell)/dc)\\
&V(m_1/(M\ell))V_1(m_2/M)e(q^{\epsilon}(m_1-m_2\ell)/ML),
\end{aligned}
\end{equation}and
\begin{equation}\notag
\begin{aligned}
\mathscr{S}=\frac{1}{qCL}\sideset{}{^*}\sum_{n\in\mathscr{N}}\alpha_n\sum_{\ell\in\mathscr{L}}\overline{\lambda(\ell)}\sum_{c\in\mathscr{C}}\frac{1}{c}\sum_{a(q)}\mathop{\sum\sum}_{\substack{m_1,m_2\geq 1}}&\lambda(m_1)K(m_2n)e(a(m_1-m_2\ell)/q)\\
&V(m_1/(M\ell))V_1(m_2/M)e(q^{\epsilon}(m_1-m_2\ell)/ML).
\end{aligned}
\end{equation}A trivial estimation of $\mathscr{S}$ yields 
\begin{equation}\notag
\mathscr{S}\ll M^2N/C.
\end{equation}Hence we can ignore the contribution  $\mathscr{S}$ since $C$ is allowed to be arbitrary large. The rest of the section is devoted to the estimation of $S(d), d|q$.
\subsection{Dualisation}
In the cuspidal case, the Voronoi summation transforms the $m_1$-sum in $S(d)$  into
\begin{equation}\label{gl2dual}
\sum_{m_1\geq 1}\lambda(m_1)e\left(\frac{am_1}{dc}\right)V(m_1/(M\ell))e(q^{\epsilon}m_1/ML)=\frac{ML}{dc}\sum_{\tilde{m_1}\geq 1}\lambda(\tilde{m_1})e\left(\frac{\pm\bar{a}\tilde{m_1}}{dc}\right)I_1^{\pm}(\tilde{m}_1,c),
\end{equation}where $I_1^{\pm}(\tilde{m}_1,c)=(ML)^{-1}H^{\pm}(\tilde{m}_1/d^2c^2)$, $H^{\pm}$ as in Lemma \ref{gl2vor}. Note that in each case, $I_1^{\pm}(\tilde{m}_1,c)$ will be roughly of the form
\begin{equation}\notag
I_1^{\pm}(\tilde{m}_1,c)\approx \text{(constant factors)}\,\,\int_{\mathbb{R}}V(y)e(q^{\epsilon}y)K\left(\frac{\sqrt{ML\tilde{m}_1y}}{dc}\right)\,dy,
\end{equation}where $K(\cdots)$ is one of the Bessel functions appearing in Lemma \ref{gl2vor}. Since the order of these Bessel functions are fixed for us, $K(x)$ will oscillate like $e(x)$ (see \cite{watson}, p. 206). Hence by  repeated integration by parts, we can conclude that $I_1^{\pm}(\tilde{m}_1,c)$ is negligibly small unless
\begin{equation}\label{truc}
\tilde{m}_1\asymp q^{\epsilon}d^{2}c^2/ML,
\end{equation}in which case the $j$-th derivative is trivially bounded by
\begin{equation}\label{der1}
|\tilde{m}_1|^j\frac{\partial^j I_1^{\pm}(\tilde{m}_1,c)}{\partial \tilde{m_1}^j}\ll_{j,\epsilon} q^{j\epsilon}.
\end{equation} The Poisson summation transforms the $m_2$-sum in \eqref{S_d} into
\begin{equation}\notag
\begin{aligned}
&\sum_{m_2\sim M}K(m_2n)e\left(\frac{-am_2\ell}{dc}\right)V_1(m_2/M)e(-q^{\epsilon}m_2\ell/ML)\\
&= \frac{M}{qc}\sum_{\alpha(qc)}K(\alpha n)e\left(\frac{-a\alpha\ell}{dc}\right)\sum_{\tilde{m_2}\in\mathbb{Z}}e\left(\frac{-\tilde{m_2}\alpha}{qc}\right)I_2(\tilde{m_2},c),
\end{aligned}
\end{equation}where
\begin{equation}\notag
I_2(\tilde{m_2},c)=\int_{\mathbb{R}}V_1(x)e\left(\frac{q^{\epsilon}\ell}{L}-\frac{M\tilde{m_2}x}{qc}\right)dx.
\end{equation}Again, from repeated integration by parts it follows that $I_2(\tilde{m_2},c)$ is negligible unless
\begin{equation}\notag
|\tilde{m}_2|\asymp q^{1+\epsilon}c/M
\end{equation}and the $j$-th derivative is bounded by
\begin{equation}\label{der2}
|\tilde{m}_2|^j\frac{\partial^j I_2(\tilde{m_2},c)}{\partial \tilde{m_2}^j}\ll_{j,\epsilon} q^{j\epsilon}.
\end{equation}

Combining the above two transformations, we see that $S(d)$ can be replaced by
\begin{equation}\label{B}
S(d)=\frac{M^2}{q^2dC}\sideset{}{^*}\sum_{n\in\mathscr{N}}\alpha_n\sum_{\ell\in\mathscr{L}}\overline{\lambda(\ell)}\sum_{c\in\mathscr{C}}\frac{1}{c^3}\sum_{\tilde{m_1}\asymp d^2c^2/ML}\sum_{|\tilde{m_2}|\asymp qc/M}\lambda(\tilde{m_1})\mathfrak{C}(\cdots)J(\tilde{m_1},\tilde{m_2},c),
\end{equation}where
\begin{equation}\notag
J(\tilde{m_1},\tilde{m_2},c)=I_1^{\pm}(\tilde{m}_1,c)I_2(\tilde{m_2},c),
\end{equation}and
\begin{equation}\label{charprimem}
\mathfrak{C}(\cdots)=\sideset{}{^*}\sum_{a (dc)}\sum_{\alpha (qc)}K(\alpha n)e\left(-\frac{a\alpha\ell}{dc}-\frac{\tilde{m_2}\alpha}{qc}\pm\frac{\overline{a}\tilde{m_1}}{dc}\right).
\end{equation}Note that from \eqref{der1} and \eqref{der2} we have
\begin{equation}\label{finalder}
|\tilde{m}_1|^{j_1}|\tilde{m}_2|^{j_2}\frac{\partial^{j_1}\partial^{j_2}J(\tilde{m_1},\tilde{m_2},c)}{\partial \tilde{m_1}^{j_1}\partial \tilde{m_2}^{j_2}}\ll_{j_1,j_2}q^{(j_1+j_2)\epsilon}.
\end{equation}
\begin{remark}\label{cuspi}In the case of Eisenstein coefficients, there is an additional `0-th' term in the right hand side of \eqref{gl2dual} which we now briefly show has a small contribution towards $B$. From Lemma \ref{gl2vor}, the main term will roughly be of the form
\begin{equation}\notag
\frac{ML}{dc} I(c),
\end{equation}where $I(c)$ is a integral transform with $I(c)\ll 1$. Hence if $S_0(d)$ denotes the contribution of the main term towards $S(d)$, then
\begin{equation}\label{b0}
S(d)\ll \frac{M^2}{q^2dC}\sideset{}{^*}\sum_{n\in\mathscr{N}}\alpha_n\sum_{\ell\in\mathscr{L}}\overline{\lambda(\ell)}\sum_{c\in\mathscr{C}}\frac{1}{c^3}\sum_{\tilde{m_2}\ll qc/M}\tilde{\mathfrak{C}}(\cdots)I(c)I_2(\tilde{m_2},c),
\end{equation}where $\tilde{\mathfrak{C}}(\cdots)$ is the simpler character sum
\begin{equation}\notag
\tilde{\mathfrak{C}}(\cdots)=\sideset{}{^*}\sum_{a (dc)}\sum_{\alpha (qc)}K(\alpha n)e\left(-\frac{a\alpha\ell}{dc}-\frac{\tilde{m_2}\alpha}{qc}\right).
\end{equation}It can be easily shown that $\tilde{\mathfrak{C}}(\cdots)\ll qc$. Trivially estimating \eqref{b0} we therefore obtain
\begin{equation}\notag
S_0(d)\ll ML/(dC).
\end{equation}From the freedom of choosing $C$, it follows that $S_0(d)$ will have a negligible contribution.
\end{remark}Let us come back to the generic case \eqref{B}. Dividing the $\tilde{m_1}$-sum into dyadic blocks $\tilde{m_1}\sim M_1\asymp d^2C^2/ML$ with the localising factors $W(\tilde{m_1}/M_1)$, we get
\begin{equation}\label{dyadicB}
S(d)\ll \sup_{\substack{M_1\asymp d^2C^2/ML\\}}S(d,M_1),
\end{equation}where
\begin{equation}\label{Breduc}
S(d,M_1)=\frac{M^2}{q^2dC}\sideset{}{^*}\sum_{n\in\mathscr{N}}\alpha_n\sum_{\ell\in\mathscr{L}}\overline{\lambda(\ell)}\sum_{c\in\mathscr{C}}\frac{1}{c^3}\sum_{\tilde{m_1}\in\mathbb{Z}}W(\tilde{m_1}/M_1)\sum_{|\tilde{m_2}|\asymp qC/M}\lambda(\tilde{m_1})\mathfrak{C}(\cdots)J(\tilde{m_1},\tilde{m_2},c).
\end{equation}
\subsection{Simplifying the character sum}Splitting the $\alpha \,\,({qc})$ sum \eqref{charprimem} using the Chinese remainder theorem and executing the modulo $c$ part, we obtain the congruence relation
\begin{equation}\notag
a=-\tilde{m_2}\overline{\ell}\overline{(q/d)}\,\, ({c}),
\end{equation}and we are left with
\begin{equation}\notag
\mathfrak{C}(\cdots)=c\cdot e\left(\frac{\pm q\overline{d}^2\ell\overline{\tilde{m_2}}\tilde{m_1}}{c}\right)\sideset{}{^*}\sum_{a (d)}\sum_{\alpha (q)}K(c\alpha n)e\left(-\frac{a\alpha\ell}{d}-\frac{\tilde{m_2}\alpha}{q}\pm\frac{\overline{c}\overline{a}\tilde{m_1}}{d}\right).
\end{equation}Substituting the definition
\begin{equation}\notag
K(c\alpha n)=\frac{1}{q}\sideset{}{^*}\sum_{\beta (q)}e\left(\frac{\beta}{q}\right)\sideset{}{^*}\sum_{l(q)}e\left(\frac{lc\alpha nb+\overline{l}\overline{\beta}}{q}\right)
\end{equation}and executing the $\alpha\,\,(q)$ sum we obtain
\begin{equation}\notag
l=\overline{cnb}(\tilde{m_2}+a(q/d)\ell)\,\, ({q}).
\end{equation}Substituting we get
\begin{equation}\label{tkloos0}
\mathfrak{C}(\cdots)=c\cdot e\left(\frac{\pm q\overline{d}^2\ell\overline{\tilde{m_2}}\tilde{m_1}}{c}\right)\sideset{}{^*}\sum_{a (d)}S(1,cnb\overline{(\tilde{m_2}+a(q/d)\ell)}; q)e\left(\frac{\overline{c}\overline{a}\tilde{m_1}}{d}\right).
\end{equation}Observe that $\mathfrak{C}(\cdots)$ is additive w.r.t. $\tilde{m_1}\,\, (c)$.
\subsection{Cauchy-Schwarz and Poisson}Applying Cauchy-Schwarz inequality to  \eqref{Breduc} keeping the $\tilde{m}_1$ sum outside and everything else inside the absolute value square, we see that 
\begin{equation}\label{cauchys}
S(d,M_1)\ll \frac{M^2}{q^2dC^3}\cdot dC/(ML)^{1/2}\cdot  \Omega^{1/2},
\end{equation}where
\begin{equation}\notag
\Omega=\sum_{\tilde{m_1}\in\mathbb{Z}}W(\tilde{m_1}/M_1)\left|\sideset{}{^*}\sum_{n\in \mathscr{N}}\alpha_n\sum_{\ell\in\mathscr{L}}\overline{\lambda(\ell)}\sum_{c\in\mathscr{C}}\sum_{|\tilde{m_2}|\asymp dC/M}e\left(\frac{\pm q\overline{d}^2\ell\overline{\tilde{m_2}}\tilde{m_1}}{c}\right)\mathfrak{C}_1(n,\tilde{m_1},\tilde{m_2},\ell,c)J(\tilde{m_1},\tilde{m_2},c)\right|^2,
\end{equation}where $\mathfrak{C}_1(\cdots)$ is $\mathfrak{C}(\cdots)$ in \eqref{tkloos0} without the first factor $c$. Opening the absolute value square we obtain
\begin{equation}\label{absopen}
\begin{aligned}
\Omega=&\sum_{\substack{\ell_1,\ell_2\sim L\\}}\lambda(\ell_1)\overline{\lambda(\ell_2)}\sideset{}{^*}\sum_{n_1,n_2\in \mathscr{N}}\alpha_{n_1}\overline{\alpha_{n_2}}\sum_{c_1,c_2\in\mathscr{C}}\,\,\sum_{|\tilde{m_2}|,|\tilde{m_3}|\asymp dC/M}\\
&\sum_{\tilde{m_1}\in\mathbb{Z}} e\left(\frac{q\overline{d}^2\ell_1\overline{\tilde{m_2}}\tilde{m_1}}{c_1}-\frac{q\overline{d}^2\ell_2\overline{\tilde{m_3}}\tilde{m_1}}{c_2}\right)\mathfrak{C}_1(n_1,\tilde{m_1},\tilde{m_2},\ell_1, c_1)\overline{\mathfrak{C}_1(n_2,\tilde{m_1},\tilde{m_3},\ell_2, c_2)}\\
&\hspace{7cm}\times J(\tilde{m_1},\tilde{m_2},c_1)\overline{J(\tilde{m_1},\tilde{m_3},c_2)}W(\tilde{m_1}/M_1).
\end{aligned}
\end{equation}A final application of the Poisson summation formula transforms the $\tilde{m_1}$-sum above into
\begin{equation}\label{m_1trans}
\begin{aligned}
&\frac{M_1}{dc_1c_2}\sum_{k(pc_1c_2)}e\left(\frac{q\overline{d}^2\ell_1\overline{\tilde{m_2}}k}{c_1}-\frac{q\overline{d}^2\ell_2\overline{\tilde{m_3}}k}{c_2}\right)\mathfrak{C}_1(n_1,k,\tilde{m_2},\ell_1, c_1)\overline{\mathfrak{C}_1(n_2,k,\tilde{m_3},\ell_2, c_2)}\\
&\hspace{7cm}\times\sum_{\tilde{m_4}\in\mathbb{Z}}e\left(\frac{-\tilde{m_4}k}{dc_1c_2}\right)\mathscr{I}(\tilde{m}_2,\tilde{m}_3,\tilde{m}_4,c_1,c_2)\\
&=\frac{M_1}{d}\sum_{\tilde{m_4}\in\mathbb{Z}}\mathfrak{C}_2(\cdots)\cdot\mathscr{I}(\tilde{m}_2,\tilde{m}_3,\tilde{m}_4,c_1,c_2)\cdot\delta_{c_2\ell_1\overline{\tilde{m_2}}-c_1\ell_2\overline{\tilde{m_3}}=\overline{(q/d)}\tilde{m_4} (c_1c_2)},
\end{aligned}
\end{equation}where
\begin{equation}\label{fint}
\mathscr{I}(\tilde{m}_2,\tilde{m}_3,\tilde{m}_4,c_1,c_2)=\int_{\mathbb{R}}W(x)J(M_1x,\tilde{m_2},c_1)\overline{J(M_1x,\tilde{m_3},c_2)}e(-M_1\tilde{m_4}x/dc_1c_2)dx
\end{equation}and
\begin{equation}\notag
\mathfrak{C}_2(\cdots)=\sum_{k(d)}\mathfrak{C}_1(n_1,k,\tilde{m_2},\ell_1, c_1)\overline{\mathfrak{C}_1(n_2,k,\tilde{m_3},\ell_2, c_2)}e\left(\frac{-\overline{c_1c_2}\tilde{m_4}k}{d}\right).
\end{equation}Equation \eqref{finalder} and repeated integration by parts in \eqref{fint}  allows us to truncate
\begin{equation}\notag
|\tilde{m_4}|\ll dC^2/M_1.
\end{equation}
Substituting the transformation \eqref{m_1trans} into \eqref{absopen} , we obtain
\begin{equation}\label{omega}
\begin{aligned}
\Omega=\frac{M_1}{d}\sum_{\substack{\ell_1,\ell_2\sim L\\}}\lambda(\ell_1)\overline{\lambda(\ell_2)}\sideset{}{^*}\sum_{n_1,n_2\in \mathscr{N}}\alpha_{n_1}\overline{\alpha_{n_2}}\sum_{c_1,c_2\in\mathscr{C}}\,\,\mathop{\sum_{\tilde{m_2},\tilde{m_3}\asymp dC/M}\,\sum_{\tilde{m_4}\ll dC^2/M_1}}_{c_2\ell_1\overline{\tilde{m_2}}-c_1\ell_2\overline{\tilde{m_3}}=\overline{(q/d)}\tilde{m_4} (c_1c_2)}\mathfrak{C}_2(\cdots)\cdot\mathscr{I}(\tilde{m}_2,\tilde{m}_3,\tilde{m}_4,c_1,c_2).
\end{aligned}
\end{equation}It remains to estimate the character sum $\mathfrak{C}_2(\cdots)$. Substituting the definition \eqref{tkloos0} and executing the $k\,\,(d)$-sum in $\mathfrak{C}_2(\cdots)$, we obtain
\begin{equation}\label{charfreduc}
\mathfrak{C}_2(\cdots)=d\mathop{\sideset{}{^*}\sum\sideset{}{^*}\sum}_{\substack{a_1, a_2 (d)\\c_2\overline{a_1}-c_1\overline{a_2}=\tilde{m_4} (d)}}S(1,c_1n_1\overline{(\tilde{m_2}+a_1(q/d)\ell_1)}; q)\overline{S}(1,c_2n_2\overline{(\tilde{m_3}+a_2(q/d)\ell_2)}; q).
\end{equation}
\begin{lemma}\label{glue}
Let $\mathfrak{C}_2(\cdots)$ as in \eqref{charfreduc}. Then
\begin{equation}\label{•}
\mathfrak{C}_2(\cdots)\ll qd^{3/2}\sum_{k|d}k^{1/2}\delta_{\left(\substack{\tilde{m_4}=0 (k)\\n_1c_1\tilde{m_3}=n_2c_2\tilde{m_2}(k)\\n_1c_1^2\ell_2=n_2c_2^2\ell_1 (k)}\right)}.
\end{equation}
\end{lemma}
\begin{proof}
Since $q$ is square-free, we can split the Kloosterman sums modulo $q/d$ and $d$ and get
\begin{equation}\label{char2}
\mathfrak{C}_2(\cdots)=d\cdot S(1,\overline{d}^2c_1n_1b\overline{\tilde{m_2}}; q/d)\overline{S}(1,\overline{d}^2c_2n_2b\overline{\tilde{m_3}}; q/d)\cdot\mathfrak{C}_3,
\end{equation}where
\begin{equation}\notag
\mathfrak{C}_3=\mathop{\sideset{}{^*}\sum\sideset{}{^*}\sum}_{\substack{a_1, a_2 (d)\\c_2\overline{a_1}-c_1\overline{a_2}=\tilde{m_4} (d)}}S(1,\overline{(q/d)}^2c_1n_1b\overline{(\tilde{m_2}+a_1(q/d)\ell_1)}; d)\overline{S}(1,\overline{(q/d)}^2c_2n_2b\overline{(\tilde{m_3}+a_2(q/d)\ell_2)}; d).
\end{equation}Suppose $d=d_1d_2\cdots d_l$, where each $d_i$ is prime. Then $\mathfrak{C}_3$ further factorises as
\begin{equation}\label{char3}
\mathfrak{C}_3=\prod_{i=1}^{l}K_i,
\end{equation}where
\begin{equation}\notag
K_i=\mathop{\sideset{}{^*}\sum\sideset{}{^*}\sum}_{\substack{a_1, a_2 (d_i)\\c_2\overline{a_1}-c_1\overline{a_2}=\tilde{m_4} (d_i)}}S(1,\overline{(q/d_i)}^2c_1n_1b\overline{(\tilde{m_2}+a_1(q/d)\ell_1)}; d_i)\overline{S}(1,\overline{(q/d_i)}^2c_2n_2b\overline{(\tilde{m_3}+a_2(q/d)\ell_2)}; d_i).
\end{equation}We apply the estimates from Lemma \ref{primecharest} with the parameters
\begin{equation}\notag
(s_1,t_1)=((q/d_i)^2(q/d)\overline{c_1n_1b}\ell_1, (q/d_i)^2\overline{c_1n_1b}\tilde{m_2}),\,\,\,(s_2,t_2)=((q/d_i)^2(q/d)\overline{c_2n_2b}\ell_2, (q/d_i)^2\overline{c_2n_2b}\tilde{m_3}),
\end{equation}
\begin{equation}\notag
(\lambda_1,\lambda_2)=(c_2,c_1),
\end{equation}and $m=\tilde{m_4}$. The congruences $m=0\,(p), t_1=t_2\,(p)$ and $\lambda_1s_1=\lambda_2s_2\,(p)$ then translates into $\tilde{m_4}=0\,(d_i), n_1c_1\tilde{m_3}=n_2c_2\tilde{m_2}\,(d_i)$ and $ n_1c_1^2\ell_1=n_2c_2^2\ell_2\,(d_i)$ respectively. Hence Lemma \ref{primecharest} gives
\begin{equation}\notag
K_i\ll d_i^{3/2}\left(1+d_i^{1/2}\delta_{\left(\substack{\tilde{m_4}=0 (d_i)\\n_1c_1\tilde{m_3}=n_2c_2\tilde{m_2}(d_i)\\n_1c_1^2\ell_2=n_2c_2^2\ell_1 (d_i)}\right)}\right).
\end{equation}Plugging in these estimates in \eqref{char3} we obtain
\begin{equation}\notag
\mathfrak{C}_3\ll d^{3/2}\prod_{i=1}^{k}\left(1+d_i^{1/2}\delta_{\left(\substack{\tilde{m_4}=0 (d_i)\\n_1c_1\tilde{m_3}=n_2c_2\tilde{m_2}(d_i)\\n_1c_1^2\ell_2=n_2c_2^2\ell_1 (d_i)}\right)}\right).
\end{equation}Since $d_i$'s are pairwise coprime, the congruences can be clubbed together to yield 
\begin{equation}\notag
\mathfrak{C}_3\ll d^{3/2}\sum_{k|d}k^{1/2}\delta_{\left(\substack{\tilde{m_4}=0 (k)\\n_1c_1\tilde{m_3}=n_2c_2\tilde{m_2}(k)\\n_1c_1^2\ell_2=n_2c_2^2\ell_1 (k)}\right)}.
\end{equation}The lemma follows after plugging the last estimate into \eqref{char2} and using the Weil's bound for the remaining two Kloosterman sums.
\end{proof}
 We proceed to estimate the contribution of the zero and non-zero frequencies in $\Omega$.
\subsubsection{The zero frequency}Assuming $(c_i,\ell_i)=1$, when $\tilde{m_4}=0$, the congruence
\begin{equation}\notag
c_2\ell_1\overline{\tilde{m_2}}-c_1\ell_2\overline{\tilde{m_3}}=\overline{(q/d)}\tilde{m_4} (c_1c_2)
\end{equation} in \eqref{omega} implies 
\begin{equation}\label{zerofreqcong1}
c_1=c_2=c\,\,\text{and}\,\,\,\ell_2\tilde{m_2}=\ell_1\tilde{m_3} (c).
\end{equation}Therefore, from Lemma \ref{glue} we get
\begin{equation}\label{charbd}
\mathfrak{C}_2(\cdots)\ll qd^{3/2}\sum_{k|d}k^{1/2}\delta_{\left(\substack{n_1\tilde{m_3}=n_2\tilde{m_2}(k)\\n_1\ell_2=n_2\ell_1 (k)}\right)}.
\end{equation}in the case of zero frequency. Let $\Omega_0$ denote the contribution of the zero frequency towards $\Omega$ \eqref{omega}.  Then from the above estimate for the character sum we get
\begin{equation}\notag
\Omega_0\ll\frac{M_1}{d}\cdot qd^{3/2}\cdot\sum_{\substack{\ell_1,\ell_2\sim L\\}}|\lambda(\ell_1)\overline{\lambda(\ell_2)}|\sum_{n_1,n_2\in \mathscr{N}}|\alpha_{n_1}\overline{\alpha_{n_2}}|\sum_{c\in\mathscr{C}}\sum_{k|d}k^{1/2}\mathop{\sum\sum}_{\substack{\tilde{m_2},\tilde{m}_3\asymp dC/M\\ \ell_2\tilde{m}_2=\ell_1\tilde{m_3}(c) }} \delta_{\left(\substack{n_1\tilde{m_3}=n_2\tilde{m_2}(k)\\n_1\ell_2=n_2\ell_1 (k)}\right)}.
\end{equation}Given $n_1,n_2,\ell_1,\ell_2$ and $\tilde{m_2}$, $\tilde{m_3}$ is determined modulo $kc$ from the two congruence conditions. Hence the number the $(\tilde{m_2},\tilde{m_3})$ pairs satisfying the congruence is at most $dC/M(1+d/(kM))$. Therefore
\begin{equation}\label{p3}
\Omega_0\ll \frac{M_1}{d}\cdot qd^{3/2}\cdot\sum_{\substack{\ell_1,\ell_2\sim L\\}}|\lambda(\ell_1)\overline{\lambda(\ell_2)}|\sum_{c\in\mathscr{C}}\sum_{k|d}\sum_{\substack{n_1,n_2\in \mathscr{N}\\n_1\ell_1=n_2\ell_2 (k)}}k^{1/2}(dC/M)(1+d/(kM)).
\end{equation}Before proceeding further, we use the inequality $|\lambda(\ell_1)\overline{\lambda(\ell_2)}|\ll |\lambda(\ell_1)|^2+|\lambda(\ell_2)|^2$, and due to symmetry we consider the contribution of first term only. Now given $(n_1,\ell_1)$, there are at most $(1+NL/k)$ many $(n_2,\ell_2)$ pairs satisfying the congruence in \eqref{p3}. Hence, 
\begin{equation}
\begin{aligned}
\Omega_0&\ll \frac{M_1}{d}\cdot qd^{3/2}\cdot\sum_{\substack{\ell_1\sim L\\}}|\lambda(\ell_1)|^2\sum_{c\in\mathscr{C}}\sum_{k|d}\sum_{\substack{n_1\in \mathscr{N}\\}}k^{1/2}(dC/M)(1+d/(kM))(1+NL/k)\\
&\ll \frac{M_1}{d}\cdot qd^{3/2}\cdot LCN(dC/M)\sum_{k|d}k^{1/2}(1+d/(kM))(1+NL/k)\\
&\ll \frac{M_1}{d}\cdot qd^{3/2}\cdot LCN(dC/M)\left(d^{1/2}+(d/M)+NL+(d/M)NL\right).
\end{aligned}
\end{equation}Note that the last three term inside the parenthesis of the last line is dominated by $NL(1+d/M)$. Substituting the upper bound $M_1\ll d^2C^2/ML$ we then obtain
\begin{equation}\label{omegazeo}
\Omega_{0}\ll qd^4NC^4/M^2+qd^{7/2}N^2LC^4(1+d/M)/M^2.
\end{equation}
\subsubsection{Non-zero frequencies}Let $\Omega_{\neq 0}$ denote the contribution of the non-zero frequencies $\tilde{m_4}\neq 0$ towards $\Omega$ \eqref{omega}. We use the estimate
\begin{equation}\notag
\mathfrak{C}_2(\cdots)\ll qd^{3/2}\sum_{k|d}k^{1/2}\delta_{\left(\substack{\tilde{m_4}=0 (k)}\right)}
\end{equation}from Lemma \ref{glue} in this case. With this bound in \eqref{omega}, we get
\begin{equation}\notag
\Omega_{\neq 0}\ll \frac{M_1}{d}\cdot qd^{3/2}\cdot\sum_{\substack{\ell_1,\ell_2\sim L\\}}|\lambda(\ell_1)\overline{\lambda(\ell_2)}|\sum_{n_1,n_2\in \mathscr{N}}\sum_{c_1,c_2\in\mathscr{C}}\sum_{k|d}\,\,\mathop{\sum_{\tilde{m_2},\tilde{m_3}\asymp dC/M}\,\sum_{\tilde{m_4}\ll dC^2/M_1}}_{c_2\ell_1\overline{\tilde{m_2}}-c_1\ell_2\overline{\tilde{m_3}}=\overline{(q/d)}\tilde{m_4} (c_1c_2)}k^{1/2}\delta_{k|\tilde{m_4}}.
\end{equation}We write $\tilde{m_4}=k\lambda, \lambda\ll dC^2/(M_1k), \lambda\neq 0$ and rewrite the above as
\begin{equation}\label{A0}
\Omega_{\neq 0}\ll \frac{M_1}{d}\cdot qd^{3/2}\cdot\sum_{\substack{\ell_1,\ell_2\sim L\\}}|\lambda(\ell_1)\overline{\lambda(\ell_2)}|\sum_{n_1,n_2\in \mathscr{N}}\sum_{c_1,c_2\in\mathscr{C}}\sum_{k|d}k^{1/2}\sum_{\lambda\ll dC^2/(M_1k)}\mathop{\sum_{\tilde{m_2}\asymp dC/M}\,\sum_{\tilde{m_3}\asymp dC/M}}_{c_2\ell_1\overline{\tilde{m_2}}-c_1\ell_2\overline{\tilde{m_3}}=\overline{(q/d)}k\lambda(c_1c_2)}1.
\end{equation}The number of pairs $(\tilde{m_2},\tilde{m_3})$ satisfying the congruence modulo $c_1c_2$ in \eqref{A0} is at most $(c_2\ell_1, k\lambda)(c_1\ell_2,k\lambda)(1+d/M)^2$. Recall that $(c_i,q)=(\ell_j,q)=1$ and consequently $(c_il_j, k)=1$. Hence
\begin{equation}\notag
\begin{aligned}
\Omega_{\neq 0}&\ll \frac{M_1}{d}\cdot qd^{3/2}(1+d/M)^2\cdot\sum_{\substack{\ell_1,\ell_2\sim L\\}}|\lambda(\ell_1)\overline{\lambda(\ell_2)}|\sum_{n_1,n_2\in \mathscr{N}}\sum_{k|d}k^{1/2}\\
&\hspace{6cm}\sum_{\lambda\ll dC^2/(M_1k)}(\ell_1,\lambda)(\ell_2,\lambda)\sum_{c_1,c_2\in\mathscr{C}}(c_2, \lambda)(c_1,\lambda).
\end{aligned}
\end{equation}We next execute the $(c_1,c_2)$-sum with the bound $C^2$, the $\lambda$-sum with bound the $(\ell_1,\ell_2)(dC^2/M_1k)$, and then the $(n_1,n_2)$-sum with the bound $N^2$. We arrive at
\begin{equation}\label{omeganzero}
\begin{aligned}
\Omega_{\neq 0}&\ll \frac{M_1}{d}\cdot qd^{3/2}(1+d/M)^2N^2C^2(dC^2/M_1)\cdot\sum_{\substack{\ell_1,\ell_2\sim L\\}}|\lambda(\ell_1)\overline{\lambda(\ell_2)}|(\ell_1,\ell_2)\\
&\ll \frac{M_1}{d}\cdot qd^{3/2}(1+d/M)^2N^2C^2(dC^2/M_1)\left(L\sum_{\ell\sim L}|\lambda(\ell)|^2+\sum_{\substack{\ell_1,\ell_2\sim L\\}}|\overline{\lambda(\ell_1)}\lambda(\ell_2)|\right)\\
&\ll qd^{3/2}C^4N^2L^2(1+d/M)^2.
\end{aligned}
\end{equation}

From \eqref{omegazeo} and \eqref{omeganzero} we get
\begin{equation}\label{omegafinal}
\begin{aligned}
\Omega=\Omega_0+\Omega_{\neq 0}&\ll qd^4NC^4/M^2+qd^{3/2}C^4N^2L^2(1+d/M)^2+qd^{7/2}N^2LC^4(1+d/M)/M^2\\
&=qd^4NC^4/M^2+qd^{7/2}C^4N^2L^2(1+M/d)^2/M^2+qd^{9/2}N^2LC^4(1+M/d)/M^3
\end{aligned}
\end{equation}

\subsection{Optimal choice for $L$} Substituting the last estimate into \eqref{cauchys} we arrive at
\begin{equation}\notag
S(d,M_1)\ll \frac{d^2M^{1/2}N^{1/2}}{q^{3/2}L^{1/2}}+\frac{d^{7/4}M^{1/2}NL^{1/2}}{q^{3/2}}(1+M/d)+\frac{d^{9/4}N}{q^{3/2}}(1+M/d)^{1/2}.
\end{equation}Therefore from \eqref{dyadicB} and \eqref{ssplit} it follows 
\begin{equation}\notag
\tilde{S}\ll \frac{q^{1/2}M^{1/2}N^{1/2}}{L^{1/2}}+q^{1/4}M^{1/2}NL^{1/2}(1+M/q)+q^{3/4}N(1+M/q)^{1/2}.
\end{equation}

Equating the first two term we obtain
\begin{equation}\notag
L=q^{1/4}N^{-1/2}(1+M/q)^{-1}.
\end{equation}Note that this choice make sense since the right hand side is $\gg 1$ due to the assumption $N\leq q^{1/2}(1+M/q)^{-2}$ in Theorem \ref{prime}. With the above choice we obtain
\begin{equation}\notag
\tilde{S}\ll M^{1/2}N^{3/4}q^{3/8}(1+M/q)^{1/2}+Nq^{3/4}(1+M/q)^{1/2}.
\end{equation}Substituting the above in \eqref{amp}, we finally obtain
\begin{equation}\notag
S\ll M^{1/2}N^{3/4}q^{3/8}(1+M/q)^{1/2}+MN^{3/2}q^{-1/4}(1+M/q)+Nq^{3/4}(1+M/q)^{1/2}.
\end{equation}.

\section{Proof of Theorem \ref{primepower}}
Here $q=p^{\gamma}, \gamma\geq 2$ and $p>2$. We proceed slightly differently in this case. Instead of using the entire modulus $q$ for the conductor lowering mechanism, we only use a part $p^r$, where $r<q$ is chosen optimally later. This serves two purposes; it simplifies certain counting arguments arising from the character sum estimates,  and more importantly, it introduces more terms in the `diagonal' while having a lesser impact in the off-diagonals as compared to the case of amplification.

Note that we can assume $(n,p)=1$ since otherwise the trace function vanishes. As earlier, let $\mathscr{C}$ be the set primes in $[C,2C]$, with $C$ such that
\begin{equation}\label{qassump}
p^rC>100M.
\end{equation}We choose a large $C$ such that $(c,q)=1$ for all $ c\in\mathscr{C}$. Due to the above relation, we can recast $S$ as
\begin{equation}\notag
S=\frac{1}{C}\sum_{c\in \mathscr{C}}\sideset{}{^*}\sum_{n\in \mathscr{N}}\alpha_n\mathop{\sum\sum}_{\substack{m_1\sim M\\m_2\sim M\\p^rc|(m_1-m_2)}}\lambda(m_1) K(m_2n)V(m_1/M)V_1(m_2/M)e(q^{\epsilon}(m_1-m_2)/M),
\end{equation}Detecting the congruence condition using additive characters we obtain
\begin{equation}\notag
\begin{aligned}
S=\frac{1}{p^{r}C}\sideset{}{^*}\sum_{n\in \mathscr{N}}\alpha_n\sum_{c\in \mathscr{C}}\frac{1}{c}\sum_{a(p^{r}c)}\mathop{\sum\sum}_{\substack{m_1\sim M\\m_2\sim M}}\lambda(m_1)K(m_2n)&e(a(m_1-m_2)/p^{r }c)\\
&V(m_1/M)V_1(m_2/M)e(q^{\epsilon}(m_1-m_2)/M).
\end{aligned}
\end{equation}Breaking the $a\,\, ({p^rc})$ sum into Ramanujan sums, we obtain the decomposition
\begin{equation}\notag
S=\sum_{0\leq k\leq r}S(k)+\mathscr{S},
\end{equation}where
\begin{equation}\label{B_k}
\begin{aligned}
S(k)= \frac{1}{p^{r}C}\sideset{}{^*}\sum_{n\in \mathscr{N}}\alpha_n\sum_{c\in \mathscr{C}}\frac{1}{c}\sideset{}{^*}\sum_{a(p^{r-k}c)}\mathop{\sum\sum}_{\substack{m_1\sim M\\m_2\sim M}}\lambda(m_1)&K(m_2n)e(a(m_1-m_2)/p^{r-k }c)\\
&V(m_1/M)V_1(m_2/M)e(q^{\epsilon}(m_1-m_2)/M),
\end{aligned}
\end{equation}and
\begin{equation}\notag
\begin{aligned}
\mathscr{S}=\frac{1}{p^{r}C}\sideset{}{^*}\sum_{n\in \mathscr{N}}\alpha_n\sum_{c\in \mathscr{C}}\frac{1}{c}\sum_{a(p^{r-k})}\mathop{\sum\sum}_{\substack{m_1\sim M\\m_2\sim M}}\lambda(m_1)&K(m_2n)e(a(m_1-m_2)/p^{r-k })\\
&V(m_1/M)V_1(m_2/M)e(q^{\epsilon}(m_1-m_2)/M)
\end{aligned}
\end{equation}Note that a trivial estimation yields
\begin{equation}\notag
\mathscr{S}\ll M^2N/C,
\end{equation}and therefore can ignored since $C$ is allowed to be arbitrary large. The rest of paper is devoted to the estimation of $S(k), 0\leq k\leq r$.

\subsection{Dualisation}Arguing similarly as in Remark \ref{cuspi}, we can assume we are in the cuspidal case. The Voronoi summation transforms the $m_1$-sum in \eqref{B_k} into
\begin{equation}\notag
\sum_{m_1\geq 1}\lambda(m_1)e\left(\frac{am_1}{p^{r-k}c}\right)V(m_1/M)e(q^{\epsilon}m_1/M)= \frac{M}{p^{r-k}c}\sum_{\tilde{m_1}\geq 1}\lambda(\tilde{m_1})e\left(\frac{\pm\bar{a}\tilde{m_1}}{p^{r-k}c}\right)I_1^{\pm}(\tilde{m_1},c),
\end{equation}where $I_1^{\pm}(\tilde{m}_1,c)=M^{-1}H^{\pm}(\tilde{m}_1/p^{2(r-k)}c^2)$, $H^{\pm}$ as in Lemma \ref{gl2vor}. Due to the same reasons as in \eqref{truc},  one can truncate $\tilde{m_1}$-sum (up to a negligible error) to $\tilde{m_1}\asymp p^{2(r-k)+2\epsilon}C^2/M$. 

With the application of the Poisson summation formula the $m_2$-sum in \eqref{B_k} becomes
\begin{equation}\notag
\sum_{m_2\geq 1}K(m_2n)e\left(\frac{-am_2}{p^{r-k}c}\right)V_1(m_2/M)= \frac{M}{p^{\gamma}c}\sum_{\alpha(p^{\gamma} c)}K(\alpha n)e\left(\frac{-a\alpha}{p^{r-k}c}\right)\sum_{\tilde{m_2}\ll p^\gamma c/M}e\left(\frac{-\tilde{m_2}\alpha}{p^\gamma c}\right)I_2(\tilde{m_2},c),
\end{equation}where
\begin{equation}\notag
I_2(\tilde{m_2},c)=\int_{\mathbb{R}}V_1(x)e(q^{\epsilon}x-M\tilde{m_2}x/p^{\gamma}c)dx.
\end{equation}One can again restrict the $\tilde{m_2}$-sum  to $\tilde{m_2}\asymp p^{\gamma+\epsilon}C/M$. 

Combining the above two transformations, we see that $S(k)$ can be replaced by
\begin{equation}\label{trans}
S(k)=\frac{M^2}{p^{\gamma+2r-k}C}\sideset{}{^*}\sum_{n\in \mathscr{N}}\alpha_n\sum_{c\in\mathscr{C}}\frac{1}{c^3}\sum_{\tilde{m_1}\asymp p^{2(r-k)}C^2/M}\sum_{\tilde{m_2}\asymp p^\gamma C/M}\lambda(\tilde{m_1})\mathfrak{C}(\cdots)J(\tilde{m}_1,\tilde{m}_2, c),
\end{equation}where
\begin{equation}\notag
J(\tilde{m_1},\tilde{m_2},c)=I_1^{\pm}(\tilde{m}_1,c)I_2(\tilde{m_2},c),
\end{equation}and
\begin{equation}\label{charpp}
\mathfrak{C}(\cdots)=\sideset{}{^*}\sum_{a (p^{r-k}c)}\sum_{\alpha (p^\gamma c)}K(\alpha n)e\left(-\frac{a\alpha}{p^{r-k}c}-\frac{\tilde{m_2}\alpha}{p^\gamma c}\pm\frac{\overline{a}\tilde{m_1}}{p^{r-k}c}\right).
\end{equation}As in \eqref{finalder},  $J(\tilde{m_1},\tilde{m_2},c)$ satisfies 
\begin{equation}\label{finalder2}
|\tilde{m}_1|^{j_1}|\tilde{m}_2|^{j_2}\frac{\partial^{j_1}\partial^{j_2}J(\tilde{m_1},\tilde{m_2},c)}{\partial \tilde{m_1}^{j_1}\partial \tilde{m_2}^{j_2}}\ll_{j_1,j_2,\epsilon}p^{(j_1+j_2)\epsilon\gamma}.
\end{equation}
Dividing the $\tilde{m_1}$-sum in \eqref{trans} into dyadic blocks $\tilde{m_1}\sim M_1\asymp p^{2(r-k)}C^2/M$ and inserting localising factor $W(\tilde{m_1}/M_1)$ we get
\begin{equation}\label{midya}
S(k)\ll \sup_{\substack{M_1\asymp p^{2(r-k)}C^2/M\\}}S(k,M_1),
\end{equation}where
\begin{equation}\label{Breducpp}
S(k,M_1)=\frac{M^2}{p^{\gamma+2r-k}C}\sideset{}{^*}\sum_{n\in \mathscr{N}}\alpha_n\sum_{c\in\mathscr{C}}\frac{1}{c^3}\sum_{\tilde{m_1}\geq 1}W(\tilde{m_1}/M_1)\sum_{\tilde{m_2}\asymp p^{\gamma}C/M}\lambda(\tilde{m_1})\mathfrak{C}(\cdots)J(\tilde{m}_1,\tilde{m}_2, c).
\end{equation}
\subsection{Simplifying the character sum} Splitting the $\alpha\,\, ({pc})$ sum in \eqref{charpp} using the Chinese remainder theorem and executing the modulo $c$ part, we obtain the congruence relation
\begin{equation}\notag
a=-\overline{p}^{\gamma-r+k}\tilde{m_2}\,\,({c'}),
\end{equation}and we are left with
\begin{equation}\notag
\mathfrak{C}(\cdots)=ce\left(\frac{\pm \overline{p}^{2(r-k)}p^\gamma\overline{\tilde{m_2}}\tilde{m_1}}{c}\right)\sideset{}{^*}\sum_{a (p^{r-k})}\sum_{\alpha (p^\gamma)}K(c\alpha n)e\left(-\frac{a\alpha}{p^{r-k}}-\frac{\tilde{m_2}\alpha}{p^\gamma}\pm\frac{\overline{c}\overline{a}\tilde{m_1}}{p^{r-k}}\right).
\end{equation}Substituting the definition
\begin{equation}\notag
K(c\alpha n)=\frac{1}{p^{\gamma}}\sideset{}{^*}\sum_{\beta (p^\gamma)}e\left(\frac{\beta}{p^\gamma}\right)\sideset{}{^*}\sum_{l(p^\gamma)}e\left(\frac{lc\alpha nb+\overline{l}\overline{\beta}}{p^{\gamma}}\right)
\end{equation}and executing the $\alpha\,\, ({p^{\gamma}})$ sum we obtain
\begin{equation}\notag
l=\overline{cnb}(\tilde{m_2}+p^{\gamma-r+k}a)\,\, ({p^{\gamma}}).
\end{equation}Substituting we get
\begin{equation}\label{tkloos}
\mathfrak{C}(\cdots)=c\cdot e\left(\frac{\pm \overline{p}^{2(r-k)}p^\gamma\overline{\tilde{m_2}}\tilde{m_1}}{c}\right)\sideset{}{^*}\sum_{a (p^{r-k})}S(1,cnb\overline{(\tilde{m_2}+p^{\gamma-r+k}a)}; p^\gamma)e\left(\frac{\overline{c}\overline{a}\tilde{m_1}}{p^{r-k}}\right).
\end{equation}
\subsection{Cauchy-Schwarz and Poisson}
Applying Cauchy-Schwarz inequality to \eqref{Breducpp} keeping the $\tilde{m}_1$ sum outside and everything else inside the absolute value square, we arrive at
\begin{equation}\label{Bk}
S(k,M_1)\ll \frac{M^2}{p^{\gamma+2r-k}C^3}\cdot p^{r-k}C/M^{1/2}\cdot \Omega^{1/2},
\end{equation}where
\begin{equation}\notag
\Omega=\sum_{\tilde{m_1}\in\mathbb{Z}}W(\tilde{m_1}/M_1)\left|\sideset{}{^*}\sum_{n\in\mathscr{N}}\alpha_n\sum_{c\in\mathscr{C}}\sum_{\tilde{m_2}\asymp p^{\gamma}C/M}e\left(\frac{\pm \overline{p}^{2(r-k)}p^\gamma\overline{\tilde{m_2}}\tilde{m_1}}{c}\right)\mathfrak{C}_1(n,c,\tilde{m_1},\tilde{m_2})J(\tilde{m}_1,\tilde{m_2},c)\right|^2,
\end{equation}where $\mathfrak{C}_1(\cdots)$ is $\mathfrak{C}(\cdots)$ in \eqref{tkloos} without the first factor $c$. Opening the absolute value square we get
\begin{equation}\label{bpoisson}
\begin{aligned}
\Omega=&\sideset{}{^*}\sum_{n_1,n_2\in \mathscr{N}}\alpha_{n_1}\overline{\alpha_{n_2}}\sum_{c_1,c_2\in\mathscr{C}}\,\,\sum_{\tilde{m_2},\tilde{m_3}\asymp p^{\gamma}C/M}\\
&\sum_{\tilde{m_1}\in\mathbb{Z}}e\left(\frac{ \overline{p}^{2(r-k)}p^{\gamma}\overline{\tilde{m_2}}\tilde{m_1}}{c_1}-\frac{ \overline{p}^{2(r-k)}p^{\gamma}\overline{\tilde{m_3}}\tilde{m_1}}{c_2}\right)\mathfrak{C}_1(n_1,c_1,\tilde{m_1},\tilde{m_2})\overline{\mathfrak{C}_1(n_2,c_2,\tilde{m_1},\tilde{m_3})}\\
&\hspace{7cm}\times J(\tilde{m_1},\tilde{m_2},c_1)\overline{J(\tilde{m_1},\tilde{m_3},c_2)}W(\tilde{m_1}/M_1).
\end{aligned}
\end{equation}A final application of the Poisson summation formula transforms the $\tilde{m_1}$ sum into
\begin{equation}\label{poissonf}
\begin{aligned}
&\frac{M_1}{p^{r-k}c_1c_2}\sum_{\beta (p^{r-k}c_1c_2)}e\left(\frac{ \overline{p}^{2(r-k)}p^{\gamma}\overline{\tilde{m_2}}\beta}{c_1}-\frac{ \overline{p}^{2(r-k)}p^{\gamma}\overline{\tilde{m_3}}\beta}{c_2}\right)\mathfrak{C}_1(n_1,c_1,\beta,\tilde{m_2})\overline{\mathfrak{C}_1(n_2,c_2,\beta,\tilde{m_3})}\\
&\hspace{7cm}\times\sum_{\tilde{m_4}\in\mathbb{Z}}e\left(\frac{-\tilde{m_4}\beta}{p^{r-k}c_1c_2}\right)\mathscr{I}(\tilde{m}_2,\tilde{m}_3,\tilde{m}_4,c_1,c_2)\\
&=\frac{M_1}{p^{r-k}}\sum_{\tilde{m_4}\in\mathbb{Z}}\mathfrak{C}_2(\cdots)\cdot\mathscr{I}(\tilde{m}_2,\tilde{m}_3,\tilde{m}_4,c_1,c_2)\cdot\delta_{c_2\overline{\tilde{m_2}}-c_1\overline{\tilde{m_3}}=\overline{p}^{2(\gamma-r+k)}\tilde{m_4} (c_1'c_2')},
\end{aligned}
\end{equation}where
\begin{equation}\label{poitrans}
\mathscr{I}(\tilde{m}_2,\tilde{m}_3,\tilde{m}_4,c_1,c_2)=\int_{\mathbb{R}}W(x)J(M_1x,\tilde{m_2},c_1)\overline{J(M_1x,\tilde{m_3},c_2)}e(-M_1\tilde{m_4}x/(p^{r-k}c_1c_2))dx
\end{equation}and
\begin{equation}\notag
\mathfrak{C}_2(\cdots)=\sum_{\beta (p^{r-k})}\mathfrak{C}_1(n_1,c_1,\beta,\tilde{m_2})\overline{\mathfrak{C}_1(n_2,c_2,\beta,\tilde{m_3})}e\left(\frac{-\overline{c_1c_2}\tilde{m_4}\beta}{p^{r-k}}\right).
\end{equation}Due to \eqref{finalder2} and repeated integration by parts, \eqref{poitrans} is negligibly small unless
\begin{equation}\notag
\tilde{m_4}\ll p^{r-k}C^2/M_1\ll M/p^{r-k}.
\end{equation}
Substituting \eqref{poissonf} in place of the $\tilde{m}_1$-sum in \eqref{bpoisson} we obtain
\begin{equation}\label{fomega}
\Omega=\frac{M_1}{p^{r-k}}\sideset{}{^*}\sum_{n_1,n_2\in \mathscr{N}}\alpha_{n_1}\overline{\alpha_{n_2}}\sum_{c_1,c_2\in\mathscr{C}}\,\,\,\,\mathop{\sum_{\tilde{m_2},\tilde{m_3}\asymp p^{\gamma}C/M}\,\,\,\sum_{\tilde{m_4}\ll M/p^{r-k}}}_{c_2\overline{\tilde{m_2}}-c_1\overline{\tilde{m_3}}=\overline{p}^{(\gamma-r+k)}\tilde{m_4} (c_1c_2)}\mathfrak{C}_2(\cdots)\cdot \mathscr{I}(\tilde{m}_2,\tilde{m}_3,\tilde{m}_4,c_1,c_2).
\end{equation}It remains to estimate $\mathfrak{C}_2$. Substituting the definition \eqref{tkloos} and executing the $\beta (p^{\gamma})$ sum, we obtain
\begin{equation}\label{charreduc}
\mathfrak{C}_2(\cdots)=p^{r-k}\mathop{\sideset{}{^*}\sum \sideset{}{^*}\sum}_{\substack{a_1,a_2 (p^{r-k})\\ c_2\overline{a_1}-c_1\overline{a_2}=m_4 (p^{r-k})}}S(1,c_1n_1b\overline{(\tilde{m_2}+p^{\gamma-r+k}a_1)}; p^\gamma)\overline{S}(1,c_2n_2b\overline{(\tilde{m_3}+p^{\gamma-r+k}a_2)}; p^\gamma).
\end{equation}

We proceed for estimating the contribution of the zero and the non-zero frequencies towards \eqref{fomega}.

\subsection{The zero frequency $\tilde{m_4}=0$}
Note that from the congruence condition in \eqref{fomega}, $\tilde{m}_4=0$ implies  $c_1=c_2=c$ and $\tilde{m_3}=\tilde{m_2}\,\, (c)$.  We write $\tilde{m_3}=\tilde{m_2}+c\lambda$, $\lambda\ll p^{\gamma}/M$.
\subsubsection*{Case 1 : $n_1\neq n_2$ or $\lambda\neq 0$.} In this case the trivial estimation of \eqref{fomega} turns out to be worse than the non-diagonal contributions in the sub-Weyl range $M\ll q^{2/3}$. Fortunately, we can overcome this by exploiting the extra cancellations in the long $\tilde{m_2}$($\asymp p^{\gamma}C/M$)-sum. Let $A_0$ denote the contribution of the case under consideration  towards \eqref{fomega}. Then
\begin{equation}\label{a0}
A_0=\frac{M_1}{p^{r-k}}\sideset{}{^*}\sum_{n_1,n_2\in \mathscr{N}}\alpha_{n_1}\overline{\alpha_{n_2}}\sum_{c\in\mathscr{C}}\sum_{\lambda\ll p^{\gamma }/M}\sum_{\tilde{m_2}\asymp p^{\gamma}C/M}\mathfrak{C}_2(\cdots)\cdot \mathscr{I}(\tilde{m}_2,\tilde{m_2}+c\lambda,0,c, c),
\end{equation}where from \eqref{charreduc}
\begin{equation}\notag
\mathfrak{C}_2(\cdots)=p^{r-k}\sideset{}{^*}\sum_{a (p^{r-k})}S(1,cn_1\overline{(\tilde{m_2}+p^{\gamma-r+k}a)}; p^\gamma)\overline{S}(1,cn_2\overline{(c\lambda+\tilde{m_2}+p^{\gamma-r+k}a)}; p^\gamma).
\end{equation}We apply  Poisson summation on the $\tilde{m_2}$-sum and observe that only zero frequency survives since the conductor is $p^{\gamma}$, whereas, the length of the $\tilde{m_2}$-sum is $p^{\gamma}C/M\ggg p^{\gamma}$ when $C$ is suitable large. Hence, the $\tilde{m_2}$-sum in \eqref{a0} becomes
\begin{equation}\label{cpoi}
\begin{aligned}
&\sum_{\tilde{m_2}\asymp p^{\gamma}C/M}\mathfrak{C}_2(\cdots)\cdot \mathscr{I}(\tilde{m}_2,\tilde{m_2}+c\lambda,0,c, c)\\
&=p^{r-k}\cdot \frac{C}{M}\sideset{}{^*}\sum_{a (p^{r-k})}\sideset{}{^*}\sum_{\alpha (p^{\gamma})}S(1,cn_1b\overline{(\alpha+p^{\gamma-r+k}a)}; p^\gamma)\overline{S}(1,cn_2b\overline{(c\lambda+\alpha+p^{\gamma-r+k}a)}; p^\gamma)\cdot\mathcal{I}(\cdots),
\end{aligned}
\end{equation}where 
\begin{equation}\notag
\mathcal{I}(\cdots)=\int_{x\sim 1}\mathscr{I}((p^{\gamma}C/M)x,(p^{\gamma}C/M)x+c\lambda,0,c, c)dx.
\end{equation}After the change of variables $cn_1b\overline{(\alpha+p^{\gamma-r+k}a)}\mapsto \alpha$, the right hand side of \eqref{cpoi} then becomes
\begin{equation}
 \frac{p^{2(r-k)}C}{M}\sideset{}{^*}\sum_{\alpha (p^{\gamma})}S(1,\alpha; p^\gamma)\overline{S}(1,\overline{n_1}n_2\alpha\overline{(\alpha\overline{n_1}\lambda+1)}; p^\gamma)\cdot\mathcal{I}(\cdots)
\end{equation}and therefore
\begin{equation}\label{cpoibd}
\sum_{\tilde{m_2}\asymp p^{\gamma}C/M}\mathfrak{C}_2(\cdots)\cdot \mathscr{I}(\tilde{m}_2,\tilde{m_2}+c\lambda,0,c, c)\ll \frac{p^{2(r-k)}C}{M}\left|\sideset{}{^*}\sum_{\alpha (p^{\gamma})}S(1,\alpha; p^\gamma)\overline{S}(1,\overline{n_1}n_2\alpha\overline{(\alpha\overline{n_1}\lambda+1)}; p^\gamma)\right|.
\end{equation}
An estimate evaluation of the character sum above can be obtained by following the proof of Lemma \ref{ppowerchar}. However, this sum has been already studied in \cite{df} and we quote them directly for simplicity.
\begin{lemma}[R. Dabrowski and B. Fisher]\label{rb}
For $a\in\mathbb{Z}_p^{\times}, b\in\mathbb{Z}_p$ and $\gamma\geq 1$,
\begin{equation}\label{rb}
\sideset{}{^*}\sum_{x (p^{\gamma})}S(1,x; p^{\gamma})\overline{S}(1,ax\overline{(bx+1)};p^{\gamma})\ll p^{3\gamma/2}p^{(\min\{\gamma, \nu_p(a-1),\nu_p(b)\})/2}.
\end{equation}
\end{lemma}This is the summary of their Theorem 3.2, Proposition 3.3 and Proposition 3.4, in case of the particular character sum in \eqref{rb}. Plugging this estimate in \eqref{cpoibd}, we obtain
\begin{equation}\notag
\sum_{\tilde{m_2}\asymp p^{\gamma}C/M}\mathfrak{C}_2(\cdots)\cdot \mathscr{I}(\tilde{m}_2,\tilde{m_2}+c\lambda,0,c, c)\ll \frac{p^{3\gamma/2+2(r-k)}C}{M}\cdot p^{(\min\{\gamma, \nu_p(n_1-n_2),\nu_p(\lambda)\})/2}.
\end{equation}Consequently, \eqref{a0} can be bounded by
\begin{equation}\label{a0f}
\begin{aligned}
A_0&\ll \frac{M_1}{p^{r-k}}\cdot \frac{p^{3\gamma/2+2(r-k)}C}{M} \sum_{c\in\mathscr{C}}\mathop{\sum_{n_1,n_2\in\mathscr{N}}\sum_{\lambda \ll p^{\gamma}/M}}_{(n_1-n_2, \lambda)\neq (0,0)}p^{(\min\{\gamma, \nu_p(n_1-n_2),\nu_p(\lambda)\})/2}\\
&\ll \frac{M_1}{p^{r-k}}\cdot \frac{p^{3\gamma/2+2(r-k)}C}{M}\cdot C \cdot (N^2p^{\gamma}/M)\\
&\ll \frac{p^{5\gamma/2+3(r-k)}C^4N^2}{M^3}.
\end{aligned}
\end{equation}

\subsubsection*{Case 2: $n_1=n_2$ and $\tilde{m_2}=\tilde{m_3}$}In this case we use the trivial estimate
\begin{equation}\notag
\mathfrak{C}_2(\cdots)=p^{r-k}\sideset{}{^*}\sum_{a (p^{r-k})}S(1,cn_1\overline{(\tilde{m_2}+p^{\gamma-r+k}a)}; p^\gamma)\overline{S}(1,cn_2\overline{(c\lambda+\tilde{m_2}+p^{\gamma-r+k}a)}; p^\gamma)\ll p^{\gamma+2(r-k)}.
\end{equation}So if $B_0$ denotes the contribution of this case towards \eqref{fomega}, then
\begin{equation}\label{b0f}
B_0\ll \frac{M_1}{p^{r-k}}\sum_{n_1\in\mathscr{N}}\sum_{c\in\mathscr{C}}\sum_{\tilde{m_2}\ll p^{\gamma}C/M}p^{\gamma+2(r-k)}\ll \frac{M_1}{p^{r-k}}\cdot NC\cdot \frac{p^{\gamma}C}{M}\cdot p^{\gamma+2(r-k)}\ll \frac{p^{2\gamma+3(r-k)}C^4N}{M^2}.
\end{equation}

Combining \eqref{a0f} and \eqref{b0f}, we obtain
\begin{equation}\label{zerofref}
\Omega_0\ll \frac{p^{2\gamma+3(r-k)}C^4N}{M^2}+\frac{p^{5\gamma/2+3(r-k)}C^4N^2}{M^3}.
\end{equation}


\subsection{Non-zero frequencies $m_4\neq 0$}
We divide the $\tilde{m_4}$-sum in \eqref{fomega} into cases according to the two parts given by Lemma \ref{ppowerchar} and denote their contribution towards \eqref{fomega} by $A_1$ for the first part, and $A_2$ for the second part. Note that $u=r-k$ satisfies the hypothesis
\begin{equation}\label{choice}
u\leq 4\gamma/5
\end{equation}
 in our final choice of $r$.
\subsubsection*{Case 1: $(r-k)/2<\gamma-(r-k)$ or $\nu_p(\tilde{m_4})<\gamma-(r-k)$}In this case, the first part of Lemma \ref{ppowerchar} gives
\begin{equation}\notag
\mathfrak{C}_2(\cdots)\ll p^{\gamma+3(r-k)/2+1/2}\cdot p^{\nu_p(\tilde{m_4})}.
\end{equation}
Substituting this in \eqref{fomega}, it follows
\begin{equation}\label{nzerof}
A_1\ll\frac{M_1}{p^{r-k}}\cdot p^{\gamma+3(r-k)/2+1/2}\sideset{}{^*}\sum_{n_1,n_2\in \mathscr{N}}\,\,\sum_{c_1,c_2\in\mathscr{C}}\,\,\,\,\mathop{\sum_{\tilde{m_2},\tilde{m_3}\asymp p^{\gamma}C/M}\,\,\,\sum_{\tilde{m_4}\ll M/p^{r-k}}}_{c_2\overline{\tilde{m_2}}-c_1\overline{\tilde{m_3}}=\overline{p}^{(\gamma-r+k)}\tilde{m_4} (c_1c_2)}p^{\nu_p(\tilde{m_4})}.
\end{equation}Next consider the $\tilde{m_2},\tilde{m_3}$ sum in \eqref{nzerof}. Given $\tilde{m_4}(\neq 0)$, there are $(c_1, \tilde{m_4})(c_2, \tilde{m_4})(1+p^{\gamma}/M)^2$ many $(\tilde{m_2},\tilde{m_3})$ pairs satisfying the congruence $\bmod\,\,c_1c_2$.  Hence
\begin{equation}\label{zeoc}
\begin{aligned}
A_1&\ll \frac{M_1}{p^{r-k}}\cdot p^{\gamma+3(r-k)/2+1/2}\sideset{}{^*}\sum_{n_1,n_2\in \mathscr{N}}\,\,\sum_{\tilde{m_4}\ll M/p^u}p^{\nu_p(\tilde{m_4})}\sum_{c_1,c_2\in\mathscr{C}}(c_1, \tilde{m_4})(c_2, \tilde{m_4})(1+p^{\gamma}/M)^2\\
&\ll \frac{M_1}{p^{r-k}}\cdot p^{\gamma+3(r-k)/2+1/2}\cdot N^2C^2\left(1+\frac{p^{\gamma}}{M}\right)^2\cdot\frac{M}{p^{r-k}}.
\end{aligned}
\end{equation}
\subsubsection*{Case 2: $(r-k)/2\geq \gamma-(r-k)$ and $\nu_p(\tilde{m_4})\geq\gamma-(r-k)$}In this case the second part of Lemma \ref{ppowerchar} applies. The condition $t_1^{-3/2}s_1\lambda_1=t_2^{-3/2}s_2\lambda_2\,\, ({p^{\gamma-u}})$ translates to
\begin{equation}\label{a2cong}
c_1n_2=c_2n_1(\overline{\tilde{m_2}}\tilde{m_3})^3 \,\, ({p^{\gamma-(r-k)}}),
\end{equation}and we have the estimate
\begin{equation}\notag
\mathfrak{C}_2(\cdots)\ll p^{\gamma+3(r-k)/2+1/2}\cdot p^{(r-k)/2}.
\end{equation}
We write $\tilde{m_4}=p^{\gamma-(r-k)}\lambda, \lambda\ll M/p^{\gamma}$. The congruence condition modulo $c_1c_2$ in \eqref{fomega} then implies
\begin{equation}\label{a2cong2}
c_2=\tilde{m_2}\lambda \,\, ({c_1})\,\,\,\text{and}\,\,\,\,c_1=\tilde{m_3}\lambda\,\, ({c_2}),
\end{equation}or in other words
\begin{equation}\label{del}
c_2=c_1\delta_1+\tilde{m_2}\lambda,\,\,\text{and}\,\,\,\, c_1=c_2\delta_2+\tilde{m_3}\lambda,
\end{equation}for some $\delta_1,\delta_2\ll 1+O\left(\frac{(|\tilde{m_2}|+|\tilde{m_3}|)\lambda}{C}\right)$. Observe that 
\begin{equation}\notag
\frac{(|\tilde{m_2}|+|\tilde{m_3}|)\lambda}{C}\ll \frac{p^{\gamma}C}{MC}\cdot\frac{M}{p^{\gamma}}\ll 1.
\end{equation}Hence $\delta_1,\delta_2$ in \eqref{del} are bounded and so 
\begin{equation}\label{del1del2}
A_2\ll\frac{M_1}{p^{r-k}}\cdot p^{\gamma+3(r-k)/2+1/2}\cdot p^{(r-k)/2}\sum_{\delta_1,\delta_2\ll 1}\sum_{\lambda\ll M/p^{\gamma}}\sum_{\tilde{m_2},\tilde{m_3}\asymp p^{\gamma}C/M}\,\,\sum_{c_1,c_2\in\mathscr{C}}\,\,\sideset{}{^{\#}}\sum_{n_1,n_2\in\mathscr{N}}1,
\end{equation}where `$\#$' denotes the restrictions \eqref{a2cong} and \eqref{del}. When $\delta_1\delta_2\neq 1$, note that \eqref{del} uniquely determines the pair $(c_1,c_2)$. Fixing $(c_1,c_2)$, the sum over $n_2$ with the restriction \eqref{a2cong} is then bounded by $(1+N/p^{\gamma-(r-k)})$, and we see that \eqref{del1del2} is
\begin{equation}\label{neq1}
\begin{aligned}
A_2&\ll\frac{M_1}{p^{r-k}}\cdot p^{\gamma+3(r-k)/2+1/2}\cdot p^{(r-k)/2}\sum_{\delta_1,\delta_2\ll 1}\sum_{\lambda\ll M/p^{\gamma}}\sum_{\tilde{m_2},\tilde{m_3}\asymp p^{\gamma}C/M}\sum_{n_1\in\mathscr{N}}\left(1+\frac{N}{p^{\gamma-(r-k)}}\right)\\
&\ll \frac{M_1}{p^{r-k}}\cdot p^{\gamma+3(r-k)/2+1/2}\cdot N\left(1+\frac{N}{p^{\gamma-(r-k)}}\right)\left(\frac{p^{\gamma}C}{M}\right)^2\frac{M}{p^{\gamma}}p^{(r-k)/2}.
\end{aligned}
\end{equation}A comparison shows that the last estimate is the second line of \eqref{zeoc} times the factor
\begin{equation}\notag
\frac{1}{N}\left(1+\frac{M}{p^{\gamma}}\right)^{-2}\left(1+\frac{N}{p^{\gamma-(r-k)}}\right)p^{3(r-k)/2-\gamma}\ll p^{5(r-k)/2-2\gamma}+\left(1+\frac{M}{p^{\gamma}}\right)^{-2}\frac{p^{3(r-k)/2}}{Np^{\gamma}}\ll 1,
\end{equation}since our choice of $r$ will satisfy (see \eqref{rchoice})
\begin{equation}\label{hypo}
p^{r}\ll \min\left\{(Np^{\gamma})^{2/3}\left(1+\frac{M}{p^{\gamma}}\right)^{4/3}, p^{4\gamma/5} \right\}.
\end{equation}Hence $A_2\ll A_1$ when $\delta_1\delta_2\neq 1$. When $\delta_1\delta_2=1$, \eqref{del} will imply $\tilde{m_2}=\pm \tilde{m_3}$. Since $\tilde{m_2}>0$, it follows $\tilde{m_2}=\tilde{m_3}, \delta_1=\delta_2=-1$. Consequently, \eqref{del}  and \eqref{a2cong} becomes
\begin{equation}\notag
c_2=-c_1+\tilde{m_2}\lambda\,\,\,\,\text{and}\,\,\,\,c_1(n_1+n_2)=n_1\tilde{m_2}\lambda \,\, ({p^{\gamma-(r-k)}}).
\end{equation}Since $(n_1\tilde{m_2},p)=1$, the number of $c_1$ satisfying the last congruence is $\ll p^{\nu_p(\lambda)}C/p^{\gamma-(r-k)}$. Hence, \eqref{del1del2} in this case becomes
\begin{equation}\notag
\begin{aligned}
A_2&\ll \frac{M_1}{p^{r-k}}\cdot p^{\gamma+3(r-k)/2+1/2}\cdot p^{(r-k)/2}\sum_{\lambda\ll M/p^{\gamma}}\sum_{\tilde{m_2}\asymp p^{\gamma}C/M}\sum_{n_1,n_2\in\mathscr{N}}p^{\nu_p(\lambda)}C/p^{\gamma-(r-k)}\\
&\ll \frac{M_1}{p^{r-k}}\cdot p^{\gamma+3(r-k)/2+1/2} N^2\left(\frac{p^{\gamma}C}{M}\right)\cdot\frac{M}{p^{\gamma}}\cdot Cp^{3(r-k)/2-\gamma}.
\end{aligned}
\end{equation}The last estimate is the second line of \eqref{zeoc} times the factor
\begin{equation}\notag
\frac{p^{\gamma}}{M}\left(1+\frac{p^{\gamma}}{M}\right)^{-2}p^{3(r-k)/2-\gamma}\ll p^{3(r-k)/2-\gamma}\ll 1,
\end{equation}where we have again invoked \eqref{hypo}.

We conclude that the non-zero frequencies are dominated by $A_1$ in \eqref{zeoc}, that is,
\begin{equation}\label{ndiagf}
\begin{aligned}
\Omega_{\neq 0}&\ll\frac{M_1}{p^{r-k}}\cdot p^{\gamma+3(r-k)/2+1/2}\cdot N^2C^2\left(1+\frac{p^{\gamma}}{M}\right)^2\cdot\frac{M}{p^{r-k}}\\
&\ll p^{\gamma+3(r-k)/2+1/2}N^2C^4\left(1+\frac{p^{\gamma}}{M}\right)^2.
\end{aligned}
\end{equation}

From \eqref{zerofref} and \eqref{ndiagf} we finally have
\begin{equation}\notag
\Omega\ll \frac{p^{2\gamma+3(r-k)}C^4N}{M^2}+\frac{p^{5\gamma/2+3(r-k)}C^4N^2}{M^3}+p^{\gamma+3(r-k)/2+1/2}N^2C^4\left(1+\frac{p^{\gamma}}{M}\right)^2.
\end{equation}Substituting the last bound into \eqref{Bk} we arrive at
\begin{equation}\label{Bnzero}
\begin{aligned}
&S(k,M_1)\\
&\ll \frac{M^2}{p^{\gamma+2r-k}C^3}\cdot \frac{p^{r-k}C}{M^{1/2}}\cdot \left(\frac{p^{2\gamma+3(r-k)}C^4N}{M^2}+\frac{p^{5\gamma/2+3(r-k)}C^4N^2}{M^3}\right.\\
&\hspace{9cm}\left.+p^{\gamma+3(r-k)/2+1/2}N^2C^4\left(1+\frac{p^{\gamma}}{M}\right)^2\right)^{1/2}\\
&\ll p^{r/2-3k/2}M^{1/2}N^{1/2}+p^{\gamma/4+r/2-3k/4}N +p^{-\gamma/2-r/4-3k/4+1/4}M^{3/2}N(1+p^{\gamma}/M).
\end{aligned}
\end{equation}
\subsection{Optimal choice for $r$.}It follows from \eqref{Bnzero} and \eqref{midya} that
\begin{equation}\label{optimise}
S\ll p^{r/2}M^{1/2}N^{1/2}+p^{\gamma/2-r/4+1/4}M^{1/2}N(1+M/p^{\gamma})+p^{\gamma/4+r/2}N.
\end{equation}Equating the first two terms we obtain
\begin{equation}\label{choiced}
p^{r}= p^{2\gamma/3+1/3}N^{2/3}(1+M/p^{\gamma})^{4/3},
\end{equation}that is,
\begin{equation}\notag
r\approx\lfloor 2/3(\gamma+1+\log_p N(1+M/p^{\gamma})^2)\rfloor.
\end{equation}But recall from \eqref{choice} that $r$ is assumed to be at most $4\gamma/5$. We choose
\begin{equation}\label{rchoice}
r=\lfloor \min\{2/3(\gamma+\log_p N(1+M/p^{\gamma})^2), 4\gamma/5\}\rfloor.
\end{equation}So the third term in \eqref{optimise} can be bounded by $p^{13\gamma/20}N$. Note that when $N\leq p^{\gamma/5}(1+M/p^{\gamma})^{-2}$,
\begin{equation}\notag
2/3(\gamma+\log_p N(1+M/p^{\gamma})^2)\leq  4\gamma/5,
\end{equation}so that \eqref{choiced} holds (upto a factor of $p^{5/3}$) and we get
\begin{equation}\notag
S\ll p^{7/12}p^{\gamma/3}M^{1/2}N^{5/6}(1+M/p^{\gamma})^{2/3}+p^{13\gamma/20}N,
\end{equation}in this case. When $N> p^{\gamma/5}(1+M/p^{\gamma})^{-2}$, we have $r=\lfloor 4\gamma/5\rfloor $ so that the second term in \eqref{optimise} dominates the first and we get
\begin{equation}\notag
S\ll p^{1/4}p^{\gamma/2-\lfloor 4\gamma/5 \rfloor/4+\epsilon}M^{1/2}N(1+M/p^{\gamma})+p^{13\gamma/20}N.
\end{equation}Combining, we have the final estimate
\begin{equation}\label{festprimep}
S\ll p^{7/12}q^{1/3}M^{1/2}N^{5/6}(1+M/q)^{2/3}+\delta_{(N>q^{1/5}(1+M/q)^{-2})}p^{1/4}q^{3/10}M^{1/2}N(1+M/q)+q^{13/20}N,
\end{equation}where $q=p^{\gamma}$.

\section{An alternative estimate}
We will use the above estimates for $N$ going upto certain threshold. For $N$ larger, we get better estimates simply by applying Cauchy-Schwarz inequality followed by Poisson summation in the $m$-sum. Recall that
\begin{equation}\label{Sn}
S=\sum_{n\in\mathscr{N}}\sum_{m\geq 1}\alpha_n\lambda(m)K(mn)V(m/M).
\end{equation} where $$K(m)=\tilde{\text{Kl}}_3(mb,q)=\frac{1}{q}\sideset{}{^*}\sum_{x(q)}e\left(\frac{mbx}{q}\right)S(1,\overline{x}; q).$$ 
\begin{lemma}\label{alt}
For $q=p^{\gamma}, \gamma\geq 1$, we have
\begin{equation}\notag
S\ll MN^{1/2}+M^{1/2}Nq^{1/4}(1+M/q)^{1/2}.
\end{equation}
\end{lemma}	
To see this, we apply Cauchy-Schwarz inequality to \eqref{Sn} keeping the $m$-sum outside to get
\begin{equation}\label{csalt}
S\ll \frac{M^{1/2}}{q}\left(\sum_{m\in\mathbb{Z}}V(m/M)\left|\sum_{n\in\mathscr{N}}\alpha_n\sideset{}{^*}\sum_{x (q)}e\left(\frac{xmnb}{q}\right)S(1,\overline{x};q)\right|^2\right)^{1/2}.
\end{equation}Opening the absolute value square and dualising the $m$-sum using the Poisson summation formula we arrive at
\begin{equation}\label{altcs}
\begin{aligned}
&\sum_{m\in\mathbb{Z}}V(m/M)\left|\sum_{n\in\mathscr{N}}\alpha_n\sideset{}{^*}\sum_{x (q)}e\left(\frac{xmnb}{q}\right)S(1,\overline{x};q)\right|^2\\
&=\frac{M}{q}\sum_{n_1,n_2\in \mathscr{N}}\alpha_{n_1}\overline{\alpha}_{n_2}\sideset{}{^*}\sum_{x_1,x_2 (q)}S(1,\overline{x_1};q)\overline{S}(1,\overline{x_2};q)\sum_{r(q)}e\left(\frac{rb(n_1x_1-n_2x_2)}{q}\right)\sum_{\tilde{m}\in\mathbb{Z} }e\left(\frac{-\tilde{m}r}{q}\right)I(\tilde{m})\\
&=M\sum_{n_1,n_2\in \mathscr{N}}\alpha_{n_1}\overline{\alpha}_{n_2}\sum_{\tilde{m}\in\mathbb{Z} }\mathcal{C}(n_1, n_2, \tilde{m})\cdot I(\tilde{m}) ,
\end{aligned}
\end{equation}where
\begin{equation}\notag
 I(\tilde{m}) =\int_{\mathbb{R}}V(x)e(-M\tilde{m}x/q)dx
\end{equation}and
\begin{equation}\label{charalt}
\mathcal{C}(n_1,n_2,\tilde{m})=\sideset{}{^*}\sum_{x (q)}S(1,\overline{x};q)\overline{S}(1,\overline{(n_1\overline{n_2}x+\overline{n_2b}\tilde{m})};q).
\end{equation}It is clear that $ I(\tilde{m})$ is negligibly small unless  $\tilde{m}\ll q/M$. 

It remains to estimate the character sum $\mathcal{C}(\cdots)$. In the case of prime power moduli, an explicit evaluation of the character sum $\mathcal{C}$ as a function of $(\tilde{m}, n_1, n_2)$ can be obtained by following the proof of Lemma \ref{ppowerchar} or otherwise. If $\alpha_n=1$, as required for our application, this evaluation can be used to non-trivially bound one of the $n_1$, $n_2$, or $\tilde{m}$-sum in \eqref{altcs} using an exponent pair estimate. However, since we are not interested in this improvement for the purposes of this paper, we use the ready-made estimates available in \cite{df}.
\begin{lemma}\label{altcharest}
For any $q\geq 1$ and $\mathcal{C}(n_1,n_2,\tilde{m})$ as in \eqref{charalt}, we have
\begin{equation}\notag
\mathcal{C}(n_1,n_2,\tilde{m})\ll q^{3/2}\sum_{k|q}k^{1/2}\delta_{\left(\substack{n_1=n_2 (k)\\\tilde{m}=0 (k)}\right)}.
\end{equation}
\end{lemma}	
\begin{proof}
Let us factorise $q$ into product of prime powers $q=\prod_{1\leq i\leq}q_i$, where $q_i=p_i^{\gamma_i}$ and $p_i$' s are prime. Then by repeated use of the well known multiplicative property of the Kloosterman sums (\cite{iwaniec}, eq. (1.59)) we get
\begin{equation}\notag
S(1,\overline{x};q)\overline{S}(1,\overline{(n_1\overline{n_2}x+\overline{n_2b}\tilde{m})};q)=\prod_{1\leq i\leq l}S(1,\overline{(q/q_i)}^2\overline{x}, q_i)\overline{S}(1,\overline{(q/q_i)}^2\overline{(n_1\overline{n_2}x+\overline{n_2}\tilde{m})}, q_i).
\end{equation}Splitting the residue classes $x(q)$ in \eqref{charalt} using the Chinese Remainder Theorem, it then follows
\begin{equation}\label{cisplit}
\mathcal{C}(n_1,n_2,\tilde{m})=\prod_{1\leq i\leq l}K_i,
\end{equation}where 
\begin{equation}\notag
\begin{aligned}
K_i&=\sideset{}{^*}\sum_{x (q_i)}S(1,\overline{(q/q_i)}^2\overline{x}, q_i)\overline{S}(1,\overline{(q/q_i)}^2\overline{(n_1\overline{n_2}x+\overline{n_2}\tilde{m})}, q_i)\\
&=\sideset{}{^*}\sum_{x (q_i)}S(1,x, q_i)\overline{S}(1,\overline{n_1}n_2x\overline{((q/q_i)^2\overline{n_1}\tilde{m}x+1)}, q_i).
\end{aligned}
\end{equation}We can now apply estimates for $K_i$ from Lemma \ref{rb} giving us
\begin{equation}\notag
K_i\ll q_i^{3/2}p_i^{(\min\{\gamma_i, \nu_{p_i}(n_1-n_2),\nu_{p_i}(b)\})/2}\ll q_i^{3/2}\sum_{k|q_i}k^{1/2}\delta_{\left(\substack{n_1=n_2 (k)\\\tilde{m}=0 (k)}\right)}.
\end{equation}The lemma follows after substituting these estimates for $K_i$ into \eqref{cisplit} and gluing the congruences.
\end{proof}
Plugging in the estimate from Lemma \ref{altcharest} into \eqref{altcs} we obtain
\begin{equation}\notag
\begin{aligned}
\sum_{m\in\mathbb{Z}}V(m/M)\left|\sum_{n\in\mathscr{N}}\alpha_n\sideset{}{^*}\sum_{x (q)}e\left(\frac{xmnb}{q}\right)S(1,\overline{x};q)\right|^2&\ll Mq^{3/2}\sum_{k|q}k^{1/2}\sum_{n_1,n_2\in \mathscr{N}}\,\,\sum_{\tilde{m}\ll q/M }\delta_{\left(\substack{n_1=n_2 (k)\\\tilde{m}=0 (k)}\right)}\\
&\ll Mq^{3/2}\sum_{k|q}k^{1/2}N(1+N/k)(1+q/Mk)\\
&\ll MNq^{3/2}\sum_{k|q}(q^{1/2}+N(1+q/M))\\
&\ll MNq^2+N^2q^{5/2}(1+M/q).
\end{aligned}
\end{equation}Final substitution into \eqref{csalt} yields
\begin{equation}\label{primealt}
S\ll MN^{1/2}+M^{1/2}Nq^{1/4}(1+M/q)^{1/2}.
\end{equation}This completes the proof Lemma \ref{alt}.

\section{The application : Proof of Theorem \ref{mainthm} }Let $q\geq 1 , a\in\mathbb{Z}$ such that $(a,q)=1$. We are interested in the asymptotic of
\begin{equation}\notag
S=\sum_{\substack{n\leq X\\n=a (q)}}d_3(n).
\end{equation}Detecting $n=a\,\, (q)$ using additive characters, we obtain
\begin{equation}\notag
S=\frac{1}{q}\sum_{\alpha (q)}\sum_{n\leq X}d_3(n)e\left(\frac{\alpha(n-a)}{q}\right).
\end{equation}Splitting into Ramanujan sums we get
\begin{equation}\label{ult}
S=\sum_{d|q}S(d),
\end{equation}where
\begin{equation}\notag
S(d)=\frac{1}{q}\sideset{}{^*}\sum_{\alpha (d)}\sum_{n\leq X}d_3(n)e\left(\frac{\alpha(n-a)}{d}\right).
\end{equation}Fix $A>0$. Choose a smooth function $w(x)$  such that $w(x)=1$ for $x\in [X^{1-\epsilon/2}, X+X^{1-\epsilon}]$ and $\text{supp}(w)\subseteq  [X^{1-\epsilon}, X+X^{1-\epsilon/2}]$ and satisfying
\begin{equation}\notag
x^{j}w^{(j)}(x)\ll_{\epsilon,j} X^{j\epsilon},
\end{equation}for $j\geq 0$. Smoothing the $n$-sum in  $S_k$ using the weight function $w$, we obtain
\begin{equation}\label{smooth}
S(d)=\frac{1}{q}\sideset{}{^*}\sum_{\alpha (d)}\sum_{n\geq 1}d_3(n)w(n)e\left(\frac{\alpha(n-a)}{d}\right)+O(X^{1-\epsilon}/q).
\end{equation}
The Voronoi summation formula \eqref{vord3} for $d_3$ transforms the $n$-sum above into
\begin{equation}\label{vortrans0}
\begin{aligned}
&\sum_{n\geq 1}d_3(n)w(n)e\left(\frac{\alpha n}{d}\right)\\
&=\frac{1}{d}\int_{0}^{\infty}P(\log y, d)w(y)dy\\
&+\frac{d}{2\pi^{3/2}}\sum_{\pm}\sum_{r|d}\sum_{m\geq 1}\frac{1}{rm}\sum_{r_1|r}\sum_{r_2|\frac{r}{r_1}}\sigma_{0,0}(r/(r_1r_2), m)S(\pm m, \overline{\alpha}; d/r)\Phi_{\pm}(mr^2/d^3).
\end{aligned}
\end{equation}Substituting into \eqref{smooth} we obtain,
\begin{equation}\label{Sksplit}
S(d)=M(d)+E(d)+O(X^{1-\epsilon}/q),
\end{equation}where
\begin{equation}\label{degterms}
M(d)=\frac{1}{qd}\left(\int_{0}^{\infty}P(\log y, d)w(y)dy\right)\sideset{}{^*}\sum_{\alpha (d)}e(-\alpha a/d)=\frac{\mu(d)}{qd}\int_{0}^{\infty}P(\log y, d)w(y)dy,
\end{equation}and
\begin{equation}\label{vortrans}
\begin{aligned}
E(d)=\frac{d^2}{2\pi^{3/2}q}\sum_{\pm}\sum_{r|d}\sum_{m\geq 1}\frac{1}{rm}\sum_{r_1|r}\sum_{r_2|\frac{r}{r_1}}\sigma_{0,0}(r/(r_1r_2), m)K_{r,d}(m)\Phi_{\pm}(mr^2/d^3),
\end{aligned}
\end{equation}where
\begin{equation}\notag
K_{r,d}(m)=\frac{1}{d}\sideset{}{^*}\sum_{\alpha (d)}e(-a\alpha/d)S(\pm m, \overline{\alpha}; d/r).
\end{equation}Write $d=d_0d_1$, where $d_0$ is the square-free and $d_1$ is the square-full part. Then note that $K_{r,d}(m)$ vanishes unless $r|d_0$ in which case we have
\begin{equation}\notag
K_{r,d}(m)=\frac{\mu(r)}{d}\sideset{}{^*}\sum_{\alpha (d/r)}e(-\overline{r}a\alpha /(d/r))S(\pm m, \overline{\alpha}; d/r).
\end{equation}Recall from \eqref{ternary} that
\begin{equation}\label{sigma}
\sigma_{0,0}(r/(r_1r_2), m)=\sum_{t|(r/(r_1r_2),\,m)}\mu(t)d_3(m/t).
\end{equation}We fix the divisor $t|(r/(r_1r_2))$ in \eqref{sigma} and push the $m$-sum in \eqref{vortrans} inside to see that
\begin{equation}\label{ed}
E(d)\ll \frac{d^2}{q}\sum_{r|d_0}\frac{1}{r^2}\sum_{t|r}\frac{1}{t}|C(d,r,t)|,
\end{equation}where
\begin{equation}\notag
C(d,r,t)=\sum_{m\geq 1}\frac{d_3(m)}{m}\tilde{\text{Kl}}_3(mb, d/r)\Phi_{\pm}(mtr^2/d^3),
\end{equation} with $b=\pm \overline{r}ta$. Now from Lemma \ref{Phiprop} it follows that the $m$-sum above is negligibly small unless $m\ll d^3/(tr^2X)$. Also, if we define
\begin{equation}\notag
\psi(m)=(\min\{mtr^2X/d^3,1\})^{-1}\Phi_{\pm}(mtr^2/d^3),
\end{equation}then from the same lemma we have
\begin{equation}\notag
y^j\psi^{(j)}(y)\ll_{j}1.
\end{equation}Hence we can write
\begin{equation}\notag
C(d,r,t)=\sum_{m\ll d^3/(tr^2X)}\frac{\min\{mtr^2X/d^3,1\}}{m}\cdot d_3(m)\tilde{\text{Kl}}_3(mb, d/r)\psi(m).
\end{equation}Dividing the $m$-sum above into dyadic blocks $m\sim Y, Y\ll d^3/(tr^2X)$, we see that
\begin{equation}\notag
C(d,r,t)\ll \frac{\min\{Ytr^2X/d^3,1\}}{Y}\sup_{Y\ll d^3/(tr^2X) }|C(d,r,t,Y)|\ll \frac{tr^2X}{d^3}\sup_{Y\ll d^3/(tr^2X) }|C(d,r,t,Y)|
\end{equation}where
\begin{equation}\notag
C(d,r,t,Y)=\sum_{m\sim Y}d_3(m)\tilde{\text{Kl}}_3(mb, d/r).
\end{equation}Substituting the last inequality into \eqref{ed} we conclude
\begin{equation}\label{edf}
E(d)\ll \frac{X}{qd}\sup_{\substack{r|d_0, t|r\\Y\ll d^3/(tr^2X)}}|C(d,r,t,Y)|.
\end{equation}

We proceed for the estimation of $C(d,r,t,Y)$. We do this by converting it into a bilinear sum as in Theorem \ref{prime} with $N\ll Y^{1/3}$ using the symmetry in the factorisation of $d_3(m)$.
Expanding $d_3(m)$ into product of three variables and introduction dyadic partition in each of the variables, we get
\begin{equation}\label{dyadicd_3}
C(d,r,t,Y)\ll \sup_{\substack{N_1,N_2, N_3>0\\N_1N_2N_3\sim Y}}\left|\sum_{n_1,n_2,n_3}\text{Kl}_3(n_1n_2n_3b, d/r)V(n_1/N_1)V(n_2/N_2)V(n_3/N_3)\right|.
\end{equation}By symmetry we can assume $N_1\leq N_2\leq N_3$. Note that this forces $N_1\ll Y^{1/3}$. Gluing $n_2n_3=m$ we obtain
\begin{equation}\label{glueing}
\sum_{n_1,n_2,n_3}\text{Kl}_3(n_1n_2n_3b, d/r)V(n_1/N_1)V(n_2/N_2)V(n_3/N_3)=\sum_{n_1\sim N_1}\sum_{m\sim Y/N_1}a(m)\text{Kl}_3(mn_1b, d/r),
\end{equation}where
\begin{equation}\notag
a(m)=\sum_{b|m}V(b/N_2)V(m/bN_3).
\end{equation}Using the Mellin inversion 
\begin{equation}\label{Mellin}
V(x)=\int_{(\sigma)}\Tilde{V}(s)x^{-s }\,ds
\end{equation}we can further write
\begin{equation}\notag
a(m)=\mathop{\int\int}\Tilde{V}(s_1)\Tilde{V}(s_2)N_2^{s_1}N_3^{s_2}m^{-s_2}\sigma_{s_2-s_1}(m)\,ds_1\,ds_2.
\end{equation}Note that since $V$ is a nice weight function, we can restrict the contour in \eqref{Mellin} to $|s|\ll X^{\epsilon}$ upto a negligible error. Feeding all these information into the right hand of \eqref{glueing} we obtain
\begin{equation}\notag
\sum_{n_1\sim N_1}\sum_{m\sim Y/N_1}a(m)\text{Kl}_3(mn_1b, d/r)\ll \sup_{|s_i|\ll X^{\epsilon}}\left|\sum_{n_1\sim N_1}\sum_{m\sim Y/N_1}\sigma_{s_1}(m)m^{s_2}\text{Kl}_3(mn_1b, d/r)\right|.
\end{equation}Substituting the last relation into \eqref{dyadicd_3} we finally obtain
\begin{equation}\label{Ckf}
C(d,r,t,Y)\ll \sup_{\substack{N\ll Y^{1/3}\\|s_i|\ll X^{\epsilon}}}\left|S(Y/N,N)\right|,			
\end{equation}where
\begin{equation}\notag
S(M,N)=\sum_{n\sim N}\sum_{m\sim M}\sigma_{s_1}(m)m^{s_2}\text{Kl}_3(mnb, d/r).
\end{equation}
We are now position to apply our estimates for bilinear sums obtained in previous sections.
\subsection{Square-free moduli}Here $d/r$ is square-free. We want to apply the estimates from Theorem \ref{prime} and Lemma \ref{alt} to $S(M,N)$ with the parameters $q=d/r, M=Y/N$. For this we need to first verify the hypothesis $N\leq q^{1/2}(1+M/q)^{-2}$ of Theorem \ref{prime}. Note that $N\gg q^{1/2}$ translates to $N\gg (d/r)^{1/2}$ which implies
\begin{equation}\notag
Y^{1/3}\gg (d/r)^{1/2}\Rightarrow (d^3/(tr^2X))^{1/3}\gg (d/r)^{1/2}\Rightarrow d\gg X^{2/3},
\end{equation}which is not the case since $d\leq q \leq X^{2/3}$ in our final choice of $q (\leq X^{1/2+1/30-\epsilon})$. Similarly $NM^2\gg q^{5/2}$ will imply $d\gg X^{7/4}$ which is also not the case. Hence the condition $N\leq q^{1/2}(1+M/q)^{-2}$ is satisfied so that from Theorem \ref{prime} and Lemma \ref{alt}  we obtain
\begin{equation}\label{primebounds}
\begin{aligned}
&S(Y/N,N)\ll\\
&\begin{cases}
Y^{1/2}N^{1/4}\mathfrak{q}^{3/8}+Y/(N^{1/4}\mathfrak{q}^{1/8})+YN^{1/4}/\mathfrak{q}^{1/4}+Y^2/(N^{1/2}\mathfrak{q}^{5/4})+Nq^{3/4}+Y^{1/2}\mathfrak{q}^{1/4}/N^{1/2},\\
Y/N^{1/2}+Y^{1/2}N^{1/2}\mathfrak{q}^{1/4}+Y\mathfrak{q}^{-1/4}.
\end{cases}
\end{aligned}
\end{equation}where $\mathfrak{q}=d/r$. Our job now is to optimally choose bounds between the two lines in  \eqref{primebounds} depending on the size of $N$. Note that since $N\ll Y^{1/3}$, the second term in the second bound of \eqref{primebounds} is $\ll Y^{2/3}\mathfrak{q}^{1/4}$. Similarly the last term in the first bound is clearly $\ll Y^{2/3}\mathfrak{q}^{1/4}$.  Also note that $Y/N^{1/2}\gg Y/\mathfrak{q}^{-1/4}$ since $N\ll \mathfrak{q}^{1/2}$ as pointed out earlier. So we can write 
\begin{equation}\notag
S(Y/N,N)\ll \sum_{i=1}^{5}A_i+Y^{2/3}\mathfrak{q}^{1/4},
\end{equation}where
\begin{equation}\notag
A_1=\min\{Y^{1/2}N^{1/4}\mathfrak{q}^{3/8}, Y/N^{1/2}\}, A_2=\min\{Y/(N^{1/4}\mathfrak{q}^{1/8}), Y/N^{1/2} \},
\end{equation}
\begin{equation}\notag
A_3=\min\{YN^{1/4}/\mathfrak{q}^{1/4}, Y/N^{1/2}\}, A_4=\min\{Y^2/(N^{1/2}\mathfrak{q}^{5/4}), Y/N^{1/2}\},
\end{equation}and
\begin{equation}\notag
A_5=\min\{N\mathfrak{q}^{3/4},Y/N^{1/2}\}.
\end{equation}
$A_1$ attains its largest value when the two terms inside the parenthesis are equal, that is when $N=Y^{2/3}/\mathfrak{q}^{1/2}$, which gives
\begin{equation}\label{argeq}
A_1\leq Y^{2/3}\mathfrak{q}^{1/4}.
\end{equation}Similarly arguing, we obtain
\begin{equation}\notag
A_2\leq Y/\mathfrak{q}^{1/4},\,A_3\leq Y/\mathfrak{q}^{1/6}, A_4\leq Y^2/\mathfrak{q}^{5/4}, A_5\leq Y^{2/3}\mathfrak{q}^{1/4}.
\end{equation}Hence
\begin{equation}\label{Sestf}
S(Y/N,N)\ll Y^{2/3}\mathfrak{q}^{1/4}+Y/\mathfrak{q}^{1/6}+Y^2/\mathfrak{q}^{5/4}.
\end{equation}Substituting this in \eqref{Ckf} and then in \eqref{edf} we obtain
\begin{equation}\label{C0est}
\begin{aligned}
E(d)&\ll \frac{X}{qd}\sup_{\substack{r|d_0, t|r\\Y\ll d^3/(tr^2X)}}(Y^{2/3}(d/r)^{1/4}+Y(d/r)^{-1/6}+Y^2(d/r)^{-5/4})\\
&\ll  \frac{X}{qd}(X^{-2/3}d^{9/4}+X^{-1}d^{17/6}+X^{-2}d^{19/4})\\
&\ll X^{1/3}q^{1/4}+q^{5/6}+X^{-1}q^{11/4},
\end{aligned}
\end{equation}where he used the upper bound $d\ll q$ in the last line. The last line of \eqref{C0est} is $O(X^{1-\epsilon}/q)$ for $q\leq X^{1/2+1/30-\epsilon}$ and therefore $E(d)\ll X^{1-\epsilon}/q$ for $q\leq X^{1/2+1/30-\epsilon}$. Hence from \eqref{Sksplit} and \eqref{ult} it follows,
\begin{equation}\label{primef}
S=\sum_{d|q}M(d)+O(X^{1-\epsilon}/q),
\end{equation}for square-free $q\leq X^{1/2+1/30-\epsilon}$.
\subsection{Prime power moduli}Here $q=p^{\gamma}, \gamma\geq 2$ and so $d/r=p^k$. Without loss of generality, we can assume $k\geq 2$ since for $k=1$, we can use the estimate \eqref{Sestf} for $S(Y/N,N)$ to arrive at the same bound \eqref{C0est}. Furthermore, when $k\geq 2$ note that $r=1$ since $r$ has to divide the square-free part of $d$ which is $1$ in this case. For $d=p^{k}, k\geq 2$, using the estimate from \eqref{festprimep} and Lemma \ref{alt}, we obtain
\begin{equation}\label{primepbounds2}
S(Y/N,N)\ll
\begin{cases}
p^{7/12}Y^{1/2}N^{1/3}d^{1/3}+p^{7/12}Y^{7/6}/(N^{1/3}d^{1/3})\\
\hspace{2cm}+p^{1/4}Y^{1/2}N^{1/2}d^{3/10}+p^{1/4}Y^{3/2}/(N^{1/2}d^{7/10})+N\mathfrak{q}^{13/20},\\
Y/N^{1/2}+Y^{1/2}N^{1/2}d^{1/4}+Y\mathfrak{q}^{-1/4},
\end{cases}
\end{equation}where $d=p^k, k\geq 2$. As earlier, we use $N\ll Y^{1/3}$ to bound the second term of second bound in \eqref{primepbounds2} by $Y^{2/3}d^{1/4}$ and ignore the third term due to the inequality $N\ll d^{1/2}$. Hence this time we get
\begin{equation}\label{SkAi}
S(Y/N,N)\ll Y^{2/3}d^{1/4}+A_1+A_2+A_3+A_4+A_5,
\end{equation}where
\begin{equation}\notag
A_1=\min\{p^{7/12}Y^{1/2}N^{1/3}d^{1/3}, Y/N^{1/2}\},\,A_2=\min\{p^{7/12}Y^{7/6}/(N^{1/3}d^{1/3}), Y/N^{1/2}\},
\end{equation}
\begin{equation}\notag
A_3=\min\{p^{1/4}Y^{1/2}N^{1/2}d^{3/10}, Y/N^{1/2} \},\,A_4=\min\{p^{1/4}Y^{3/2}/(N^{1/2}d^{7/10}), Y/N^{1/2}\},
\end{equation}and
\begin{equation}\notag
A_5=\min\{N\mathfrak{q}^{13/20}, Y/N^{1/2}\}.
\end{equation}Arguing as in \eqref{argeq} we obtain the following estimates for $A_i$ :
\begin{equation}\notag
A_1\leq Y^{7/10}d^{1/5}p^{7/20},\,A_2\leq Y^{3/2}d^{-1}p^{7/4},\,A_3\leq Y^{3/4}d^{3/20}p^{1/8},\,A_4\leq Y^{3/2}d^{-7/10}p^{1/4}
\end{equation}and
\begin{equation}\notag
A_5\leq Y^{2/3}d^{13/60}.
\end{equation}Using these estimates for $A_i$ in \eqref{SkAi} we obtain
\begin{equation}\notag
S(Y/N,N)\ll Y^{2/3}d^{1/4}+Y^{7/10}d^{1/5}p^{7/20}+Y^{3/2}d^{-1}p^{7/4}+ Y^{3/4}d^{3/20}p^{1/8}+Y^{3/2}d^{-7/10}p^{1/4}+Y^{2/3}d^{13/60}
\end{equation}Substituting the last estimate into \eqref{Ckf} and then in \eqref{edf} we obtain
\begin{equation}\label{C1est}
\begin{aligned}
\begin{aligned}
E(d)&\ll \frac{X}{qd}\sup_{Y\ll d^3/X}\left(Y^{2/3}d^{1/4}+Y^{7/10}d^{1/5}p^{7/20}+Y^{3/2}d^{-1}p^{7/4}+ Y^{3/4}d^{3/20}p^{1/8}\right.\\
&\hspace{7cm}\left. + Y^{3/2}d^{-7/10}p^{1/4}+Y^{2/3}d^{13/60}\right)\\
&\ll\frac{X}{qd}\left(X^{-2/3}d^{9/4}+X^{-7/10}d^{23/10}p^{7/20}+X^{-3/2}d^{7/2}p^{7/4}+X^{-3/4}d^{12/5}p^{1/8}\right.\\
&\hspace{7cm}\left.+X^{-3/2}d^{19/5}p^{1/4}+X^{-2/3}d^{133/60}\right)\\
&=X^{1/3}q^{1/4}+X^{3/10}q^{3/10}p^{3/20}+X^{-1/2}q^{3/2}p^{7/4}+X^{1/4}q^{2/5}p^{1/8}\\
&\hspace{7cm}+X^{-1/2}q^{9/5}p^{1/4}+X^{1/3}q^{13/60}.
\end{aligned}
\end{aligned}
\end{equation}The last line in \eqref{C1est} is $O(X^{1-\epsilon}/q)$ when $q\ll X^{1/2+1/30-\epsilon}$ and $\gamma\geq 28$. The exponent $1/2+1/30$ and the power $\gamma\geq 28$ is determined by the `$X^{1/3}q^{1/4}$' and the `$X^{3/10}q^{3/10}p^{3/20}$' terms respectively.
\begin{remark}\label{improv1}
The main contributing term `$X^{1/3}q^{1/4}$'  originates from the `$Y^{1/2}N^{1/2}d^{1/4}$' term in the second line of \eqref{primepbounds2}. Thus, it is evident that any improvement in this term, which corresponds to the off-diagonal contribution in \eqref{altcs}, would result in an improvement in the exponent of distribution.
\end{remark}
Hence from \eqref{Sksplit} and \eqref{ult} it follows,
\begin{equation}\label{powerf}
S=\sum_{d|q}M(d)+O(X^{1-\epsilon}/q),
\end{equation}for $q=p^{\gamma}\leq X^{1/2+1/30-\epsilon}$ and $\gamma\geq 28$.

Finally, from \eqref{primef} and \eqref{powerf} it follows
\begin{equation}\label{pf}
S=\sum_{\substack{n\leq X\\n=a (q)}}d_3(n)=\sum_{d|q}M(d)+O(X^{1-\epsilon}/q),
\end{equation}for $q\leq X^{1/2+1/30-\epsilon}$, where $q$ is either square-free or $q=p^{\gamma}, \gamma\geq 28$. Note that the $M(d)$'s, which are given by \eqref{degterms}, are independent of the residue class $a\,\, (q)$. Hence summing the expression \eqref{pf} over all the co-prime residue classes $a\,\,(q)$, we obtain
\begin{equation}\notag
\sum_{\substack{n\leq X\\(n,q)=1}}d_3(n)=\phi(q)\left(\sum_{d|q}M(d)\right)+O(\phi(q)X^{1-\epsilon}/q),
\end{equation}from which it follows
\begin{equation}\notag
\sum_{d|q}M(d)= \frac{1}{\phi(q)}\sum_{\substack{n\leq X\\(n,q)=1}}d_3(n)+O(X^{1-\epsilon}/q).
\end{equation}Theorem \ref{mainthm} follows after substituting the last expression for $\sum_{d|q}M(d)$ into \eqref{pf}.

\section*{Acknowledgements}
The author would like to thank the Alfr\'{e}d R\'{e}nyi Institute of Mathematics for providing an excellent research environment.

\bibliographystyle{abbrv}
\bibliography{references}

\end{document}